%% file: main.tex
\documentclass[11pt]{article}
\usepackage[margin=1in]{geometry}
\input{preamble-vs.tex}
\usepackage{tikz}
\usetikzlibrary{positioning,arrows.meta}

\newlist{condlist}{enumerate}{1}
\setlist[condlist]{label=\subscript{\textbf{\textsf{C}}}{\textsf{{\arabic*}}},leftmargin=*, itemsep=0pt}

\newtheorem{assumption}{Assumption}[section]

\allowdisplaybreaks

\usepackage{natbib}

\newcommand{\FuncIneq}{\mathsf{FuncIneq}}
\newcommand{\PI}[1]{\mathsf{PI}\left(#1\right)}
\newcommand{\LSI}[1]{\mathsf{LSI}\left(#1\right)}
\newcommand{\PhiSI}[1]{\Phi\mathsf{SI}\left(#1\right)}
\newcommand{\Phient}[3]{
    \ensuremath{\if$#3$\bbJ_{#1}^{#2}\else\bbJ_{#1}^{#2}\left[#3\right]\fi}}
\newcommand{\cssf}[1]{\calC_{c}^{\infty}(#1)}
\newcommand{\var}{\mathsf{var}}
\newcommand{\ent}{\mathsf{ent}}

\makeatletter
\def\@maketitle{%
  \newpage
  \begin{center}%
  \let \footnote \thanks
    {\Large \bf \@title \par}%
  \end{center}%
  \par
  \vskip 0.5em}
\makeatother

\title{Two-scale criteria for Poincar\'{e} and log-Sobolev inequalities\\with applications to Markov chain Monte Carlo}

\begin{document}
\maketitle

\begin{center}
{\large
\begin{tabular}{c}
    \makecell{Vishwak Srinivasan\(^\dagger\)\\{\normalsize\texttt{vishwaks@mit.edu}}}
\end{tabular}
\vskip 0.5em

\normalsize
\begin{tabular}{c}
\({}^{\dagger}\)Department of Electrical Engineering and Computer Science, MIT
\end{tabular}
}
\end{center}

\begin{abstract}
    Given a collection of distributions \(\{P_{y}\}\) and a mixing distribution \(\rho\) supported over \(\bbR^{d}\), we propose new sufficient conditions under which the mixture / joint distribution satisfies a Poincar\'{e} or log-Sobolev inequality.
    We develop these sufficient conditions in a unified manner using the framework of \(\Phi\)-Sobolev inequalities \citep{chafai2004entropies}.
    The conditions that we develop in this work are satisfied by a variety of Markov chains, and consequently allows us to characterise the evolution of these functional inequalities for iterates generated by simulating these Markov chains.
    As a result, we obtain an clean error analysis for estimating a broad class of functionals using Markov chain Monte Carlo strategies along these Markov chains.
\end{abstract}

\input{introduction.tex}
\input{background.tex}
\input{two-scale-results.tex}
\input{mcmc-applications.tex}
\input{proofs.tex}

\section*{Acknowledgements}

We would like to thank Andre Wibisono for his invaluable comments towards helping improve this draft, and several technical discussions about the results in this work.
We also would like to thank Siddharth Mitra, Ashia Wilson, and Sam Power for helpful remarks.
This work was supported by a Simons Foundation Collaboration on Theory of Algorithmic Fairness Grant, and Generali.

\bibliography{refs.bib}
\bibliographystyle{plainnat}

\input{appendix}

\end{document}

%% file: preamble-vs.tex
\usepackage{amsmath}
\usepackage{amssymb}
\usepackage{amsthm}
\usepackage{forloop}
\usepackage{nicefrac}
\usepackage{mathtools}
\usepackage[table]{xcolor}
\usepackage{url}
\usepackage{bm}
\usepackage{float}
\usepackage[linesnumbered, ruled]{algorithm2e}
\usepackage{hyperref}
\hypersetup{
    colorlinks,
    linkcolor={blue!100!black!70},
    citecolor={blue!75!black},
    urlcolor={blue!75!black}
}
\usepackage[shortlabels]{enumitem}
\usepackage[nameinlink,capitalise]{cleveref}
\usepackage{cancel}
\usepackage{multirow}
\usepackage{makecell}
\usepackage{caption}
\usepackage{subcaption}
\usepackage{framed}
\usepackage{wrapfig}
\usepackage{booktabs}

\newtheorem{theorem}{Theorem}[section]
\newtheorem{lemma}{Lemma}[section]

\newtheorem{definition}{Definition}[section]
\newtheorem{proposition}{Proposition}[section]

\crefname{definition}{Def.}{Defs.}

\newcommand{\defcal} [1]{\expandafter\newcommand\csname cal#1\endcsname{{\cal #1}}}
\newcommand{\defsf} [1]{\expandafter\newcommand\csname sf#1\endcsname{{\sf #1}}}
\newcommand{\defbf} [1]{\expandafter\newcommand\csname bf#1\endcsname{{\bf #1}}}
\newcommand{\defbb} [1]{\expandafter\newcommand\csname bb#1\endcsname{{\mathbb{#1}}}}
\newcommand{\deffrak} [1]{\expandafter\newcommand\csname frak#1\endcsname{{\mathfrak{#1}}}}

\newcounter{ct}
\forLoop{1}{26}{ct}{
    \edef\letter{\Alph{ct}}
    \expandafter\defcal\letter
    \expandafter\defsf\letter
    \expandafter\defbf\letter
    \expandafter\defbb\letter
    \expandafter\deffrak\letter
}
\forLoop{1}{26}{ct}{
    \edef\letter{\alph{ct}}
    \expandafter\defcal\letter
    \expandafter\defsf\letter
    \expandafter\defbf\letter
    \expandafter\defbb\letter
    \expandafter\deffrak\letter
}
\newcommand{\argmin}{\mathop{\mathrm{argmin}}}

\newcommand{\rmd}{\mathrm{d}}

\newcommand{\rmI}{\mathrm{I}}
\newcommand{\op}{\mathrm{op}}
\newcommand{\numberthis}{\addtocounter{equation}{1}\tag{\theequation}}

\newcommand{\mirrorfunc}[1]{
    \ensuremath{\if$#1$\phi \else \phi(#1)\fi}
}
\newcommand{\mirrorfuncdual}[1]{
    \ensuremath{\if$#1$\phi^{*}\else \phi^{*}(#1)\fi}
}

\newcommand{\defeq}{\overset{\mathrm{def}}{=}}

\newcommand{\subscript}[2]{$#1 _ #2$}
\newlist{assumplist}{enumerate}{1}
\setlist[assumplist]{label=(\subscript{\textbf{A}}{{\arabic*}})}

\newcommand{\potential}[1]{
    \ensuremath{\if$#1$V\else V(#1)\fi}
}

\makeatletter
\newcommand*{\algotitle}[2]{%
  \stepcounter{algocf}%
  \hypertarget{algocf.title.\theHalgocf}{}%
  \NR@gettitle{#1}%
  \label{#2}%
  \addtocounter{algocf}{-1}%
}
\makeatother

%% file: introduction.tex
\section{Introduction}

At the core of modern statistical inference and estimation systems underlies the fundamental algorithmic task of \emph{sampling}, wherein the goal is to draw samples from a given distribution \(\pi^{\star}\).
One of the most essential downstream purpose of these samples is in estimating functionals \(F^{\star}(f, \pi^{\star}) := \bbE_{x \sim \pi^{\star}}[f(x)]\) of the distribution when exact computation of \(F^{\star}\) is intractable.
Markov chain Monte Carlo (MCMC) has proven to be an effective strategy for this estimation task over the past half-century \citep{brooks2011handbook}, especially in high-dimensional settings.
To apply MCMC, one needs to design a Markov chain \(\bfP\) and simulate it.
The Markov chain defines how the next iterate \(x'\) is generated from the current iterate \(x\), and the Markov chain is developed such that after a sufficiently long simulation, we can obtain a sample that approximately resembles one from \(\pi^{\star}\) in distribution.
Multiple such samples can be collected and aggregated to therefore obtain an estimate of \(F^{\star}\); more precisely, with \(N\) such samples \(\calS_{N} := \{x^{(i)}\}_{i=1}^{N}\), one has the empirical average
\begin{equation*}
    \widehat{F}(f, \calS_{N}) := \frac{1}{N}\sum_{x \in \calS_{N}} f(x)
\end{equation*}
to approximate \(F^{\star}\).
Naturally, the quality of the estimate has to be understood to make recommendations for an appropriate \(N\) and how \(\calS_{N}\) is collected.

We assume that there exists a stationary distribution \(\rho^{\star}\) for \(\bfP\) that need not coincide with \(\pi^{\star}\) due to a potentially irreducible bias (this is also referred to as a \emph{biased limit}).
The error between \(\widehat{F}(f, \calS_{N})\) and \(F^{\star}(f, \pi^{\star})\) can be decomposed into three key terms as shown below.
\begin{align*}
    \mathsf{Err}(\calS_{N}, f, \pi^{\star}) &:= \widehat{F}(f, \calS_{N}) - F^{\star}(f, \pi^{\star}) \\
    &= \underbrace{\widehat{F}(f, \calS_{N}) - \bbE[\widehat{F}(f, \calS_{N})]}_{\text{Concentration}} + \underbrace{\bbE[\widehat{F}(f, \calS_{N})] - F^{\star}(f, \rho^{\star})}_{\text{Convergence}} + \underbrace{F^{\star}(f, \rho^{\star}) - F^{\star}(f, \pi^{\star})}_{\text{Bias}}~.\numberthis\label{eq:error-decomp}
\end{align*}
Above, the expectation is taken w.r.t. to the randomness in the simulation procedure.

In this work, focusing on Markov chains defined over the \(d\)-dimensional Euclidean space \(\bbR^{d}\), we propose general and verifiable criteria for \(\bfP\), which when satisfied in conjunction with a \emph{functional inequality} such as the Poincar\'{e} or log-Sobolev inequality, provides a systematic way of understanding each of the contributing factors to the error above.
Functional inequalities such as the Poincar\'{e} and log-Sobolev inequalities play a central role in understanding each of these quantities, since they are essential in quantitatively characterising the concentration of measure phenomemon \citep{ledoux2001concentration} and the behaviours of Markov processes \citep{bakry2014analysis}, and play a direct role in understanding each of the three terms above.
This motivates identifying cases in which such inequalities hold, and characterising such cases through verifiable criteria.

\subsection{Our contributions to two-scale criteria}
\label{sec:intro-two-scale}

Let \(\rho\) be a probability measure over \(\bbR^{d}\) that satisfies a functional inequality \(\FuncIneq{}\), and let \(\bfP = \{P_{X | Y = y}\}_{y \in \bbR^{d}}\) be a collection of probability measures each of which are supported on \(\calX\) and satisfy \(\FuncIneq{}\) as well.
We pose the following two questions.
\begin{itemize}
    \item [\textbf{Q1}] When does the joint distribution \(\nu\) with density \(\rmd\nu(x, y) = \rmd P_{X | Y = y}(x) \cdot \rmd\rho(y)\) over \(\calX \times \bbR^{d}\) satisfy \(\FuncIneq{}\)?
    \item [\textbf{Q2}] When does the mixture distribution \(\mu\) with density \(\rmd \mu(x) = \int_{\calY} \rmd P_{X | Y = y}(x) \cdot \rmd \rho(y)\) over \(\calX\) satisfy \(\FuncIneq{}\)?
\end{itemize}

When \(\FuncIneq{}\) is the Poincar\'{e} or log-Sobolev inequality, answering \textbf{Q1} above automatically yields an answer to \textbf{Q2}.
This is because the Poincar\'{e} and log-Sobolev inequalities for the joint probability measures imply the same respectively for its marginals.
These questions are inherently not new and have been investigated in the past, focusing on the setting where \(\FuncIneq{}\) is the log-Sobolev inequality.
\citet{blower2005concentration,grunewald2009two} focus on \textbf{Q1} specifically, and concurrent work by \citet{otto2007new} provide a stronger result for \textbf{Q2} than that implied by \textbf{Q1} and marginalising the joint distribution \(\nu\).
Fundamentally, these results consider assumptions on each \(P_{X | Y = y} \in \bfP\) in addition to \(\rho\) and \(\bfP\) satisfying the aforementioned functional inequalities, and these assumptions are collectively referred to as ``two-scale criteria''\footnote{The two scales refer to \(\rho\) (macro-scale) and \(P_{X | Y = y} \in \bfP\) (micro-scale).}.
Recently, such two-scale criteria have been revisited in the papers by \citet{mou2019sampling,ge2020simulated} for the setting where \(\FuncIneq{}\) is the Poincar\'{e} inequality.

In \Cref{sec:two-scale-criteria} of this work, we identify criteria on the family of densities of \(\bfP\) -- novel to the best of our knowledge -- that answer \textbf{Q1} and \textbf{Q2} described above.
The two key features of these criteria are: (a) they imply the individual assumptions made in \citet{otto2007new,mou2019sampling,ge2020simulated}, and (b) they are obtained through a unified persective of these functional inequalities as instance of a \(\Phi\)-Sobolev inequality introduced by \citet{chafai2004entropies}
Technically, the manner in which we establish our results is more elementary in comparison to \citet{otto2007new}; they employ a Markov semigroup theory approach in their intermediate steps to arrive at their results.
Through our proof, we also additionally recover results about product and convolution probability measures formed by two independent distributions.

\subsection{Understanding MCMC estimation error}

We revisit our motivating problem in understanding the error in MCMC estimation in \Cref{sec:applications} of this work.
We show two key consequences of the criteria that we propose in the previous section.
Firstly, the criteria implies quantifiable bounds on Ollivier-Ricci curvature of \(\bfP\) \citep{ollivier2009ricci} which lead to non-asymptotic mixing time results for \(\bfP\) in the 2-Wasserstein metric.
Secondly, by answering \textbf{Q2}, we obtain an understanding of the evolution of functional inequalities along successive applications of \(\bfP\).
These directly contribute to control of the ``Concentration'' and ``Convergence'' terms in the error decomposition (\cref{eq:error-decomp}), and help differentiate strategies for collecting \(\calS_{N}\).

%% file: background.tex
\section{Background}

\paragraph{Notation}
The Euclidean space in \(d\) dimensions is represented by \(\bbR^{d}\).
For any two vectors \(v, w \in \bbR^{d}\), \(\langle v, w\rangle\) denotes the Euclidean inner product, and the norm of \(v\) is given by \(\|v\| = \sqrt{\langle v, v\rangle}\).
The density of a distribution \(\pi\) over \(\Omega \subseteq \bbR^{d}\) w.r.t. the Lebesgue measure (if it exists) is denoted by \(\rmd\pi\).
Distributions with density w.r.t. the Lebesgue measure over \(\bbR^{d}\) is denoted by \(\calP_{\mathrm{ac}}(\bbR^{d})\).
We use the shorthand notation \(\bbE_{\pi}[f]\) for the expected value of \(f(x)\) where \(x \sim \pi\).
For a function \(\psi : \bbR^{d} \times \bbR^{d} \to \bbR\)  that is differentiable in both arguments, we use \(\nabla_{1}\psi(x, y)\) to denote the gradient of \(x \mapsto \psi(x, y)\), and \(\nabla_{2}\psi(x, y)\) to denote the gradient of \(y \mapsto \psi(x, y)\).
The composition of two functions \(f, g\) where the image of \(g\) is contained in the domain of \(f\) is denoted by \(f(g)\).

\subsection{Functional inequalities}

Let \(\pi\) be a probability measure over \(\bbR^{d}\) and \(\cssf{\bbR^{d}}\) be the set of compactly supported smooth functions over \(\bbR^{d}\).
Consider a function \(f\) such that \(\bbE_{\pi}[f^{2}] < \infty\).
The variance of \(f\) with respect to \(\pi\) is defined as
\begin{equation*}
    \var{}_{\pi}[f] \defeq \bbE_{\pi}[f^{2}] - (\bbE_{\pi}[f])^{2}~.
\end{equation*}

For a positive function \(f\) such that \(\bbE_{\pi}[f \cdot \log (f)] < \infty\), the entropy of \(f\) with respect to \(\pi\) is
\begin{equation*}
    \ent{}_{\pi}[f] \defeq \bbE_{\pi}[f \cdot \log f] - \bbE_{\pi}[f] \cdot \log\left(\bbE_{\pi}[f]\right)~.
\end{equation*}

Both the variance and the entropy are instances of a more general \(\Phi\)-entropy \citep{chafai2004entropies}, which we introduce next.
Let \(\Phi : \calS \to \bbR\) be a convex function over a convex subset \(\calS\) or \(\bbR\).
For \(f\) whose range is \(\calS\) and satisfies \(\bbE_{\pi}[\Phi(f)] < \infty\), the \(\Phi\)-entropy of \(f\) with respect to \(\pi\) is defined as
\begin{equation}
\label{eq:phi-entropy}
\Phient{\pi}{\Phi}{f} \defeq \bbE_{\pi}[\Phi(f)] - \Phi\left(\bbE_{\pi}[f]\right)~.
\end{equation}

With these definitions, we define the Poincar\'{e} and log-Sobolev inequalities in the Euclidean setting.
\begin{definition}[\(\PI{\gamma}\)]
Let \(\pi\) be a probability measure over \(\bbR^{d}\).
Then, \(\pi\) is said to satisfy a \emph{Poincar\'{e} inequality with constant \(\gamma > 0\)} if for all \(f \in \cssf{\bbR^{d}}\) that are suitably integrable,
\begin{equation*}
    \var{}_{\pi}[f] \leq \gamma \cdot \bbE_{\pi}[\|\nabla f\|^{2}]~.
\end{equation*}
\end{definition}

\begin{definition}[\(\LSI{\gamma}\)]
Let \(\pi\) be a probability measure over \(\bbR^{d}\).
Then, \(\pi\) is said to satisfy a \emph{log-Sobolev inequality with constant \(\gamma > 0\)} if for all \(f \in \cssf{\bbR^{d}}\) that are suitably integrable,
\begin{equation*}
    \ent{}_{\pi}[f^{2}] \leq 2\gamma \cdot \bbE_{\pi}[\|\nabla f\|^{2}]~.
\end{equation*}
\end{definition}

Similarly, the Poincar\'{e} and the log-Sobolev inequalities can be viewed as instances of a more general class of inequalities called \(\Phi\)-Sobolev inequalities \citep{chafai2004entropies} which is defined next.

\begin{definition}[\(\PhiSI{\gamma}\)]
Let \(\pi\) be a probability measure over \(\bbR^{d}\).
Then, \(\pi\) is said to satisfy a \(\Phi\)-Sobolev inequality with constant \(\gamma > 0\) if for all \(f \in \cssf{\bbR^{d}} \cap \{\phi : \bbR^{d} \to \calS\}\),
\begin{equation*}
    \Phient{\pi}{\Phi}{f} \leq \frac{\gamma}{2} \cdot \bbE_{\pi}[\Phi''(f) \cdot \|\nabla f\|^{2}]~.
\end{equation*}
\end{definition}

The purpose of introducing the \(\Phi\)-entropy and the \(\Phi\)-Sobolev inequality is to minimise redundancies in the computations made in the note.
This is due to the following: whence
\begin{itemize}[leftmargin=*, itemsep=0pt]
    \item \(\calS = \bbR\) and \(\Phi(t) = t^{2}\), \(\Phient{\pi}{\Phi}{} = \var{}_{\pi}\) and \(\PhiSI{\gamma}\) is \(\PI{\gamma}\), and
    \item \(\calS = (0, \infty)\) and \(\Phi(t) = t\log(t)\), \(\Phient{\pi}{\Phi}{} = \ent{}_{\pi}\) and \(\PhiSI{\gamma}\) is \(\LSI{\gamma}\)~.
\end{itemize}
Another setting of \(\Phi\) of interest is \(\Phi(t) = t^{\nicefrac{2}{p}}\) for \(p \in [1, 2)\).
The resulting \(\Phi\)-Sobolev inequality is related to the inequality of \citet{latala2007between}.

\subsection{Distances between probability measures, Ollivier-Ricci curvature}

In this work, we consider certain measures of discrepancy over the space of probability measures: distances such as the \(1\)-Wasserstein distance and the total variation distance, and divergences such as the \(\mathsf{KL}\) divergence and the family of R\'{e}nyi divergences.

\begin{definition}[\(\mathsf{W}_{p}\)]
Let \(\rho_{1}, \rho_{2}\) be probability measures over \(\bbR^{d}\).
The \(p\)-Wasserstein distance between \(\rho_{1}\) and \(\rho_{2}\) is defined as
\begin{equation*}
    \mathsf{W}_{p}(\rho_{1}, \rho_{2}) := \inf_{\varsigma} \left\{\bbE_{(\sfx_{1}, \sfx_{2}) \sim \varsigma}\left[\|\sfx_{1} - \sfx_{2}\|^{p}\right]\right\}^{\nicefrac{1}{p}}
\end{equation*}
where the minimiser is taken over all possible joint distributions on \(\bbR^{d} \times \bbR^{d}\) such that its two marginals coincide with \(\rho_{1}\) and \(\rho_{2}\).
\end{definition}

\begin{definition}[\(\mathsf{TV}\)]
Let \(\rho_{1}, \rho_{2}\) be probability measures over \(\bbR^{d}\).
The total variation distance between \(\rho_{1}\) and \(\rho_{2}\) is defined as
\begin{equation*}
    \mathsf{TV}(\rho_{1}, \rho_{2}) := \sup_{A \subseteq \bbR^{d}} \rho_{1}(A) - \rho_{2}(A)~.
\end{equation*}
If \(\rho_{1}, \rho_{2} \in \calP_{\mathrm{ac}}(\bbR^{d})\), this is equivalent to
\begin{equation*}
    \mathsf{TV}(\rho_{1}, \rho_{2}) = \frac{1}{2}\int |\rmd \rho_{1}(x) - \rmd \rho_{2}(x)|~.
\end{equation*}
\end{definition}

\begin{definition}[\(\mathsf{KL}\)]
Let \(\rho_{1}, \rho_{2}\) be probability measures over \(\bbR^{d}\) such that \(\rho_{1}\) is absolutely continuous w.r.t. \(\rho_{2}\).
The Kullback-Leibler divergence between \(\rho_{1}\) and \(\rho_{2}\) is defined as
\begin{equation*}
    \mathsf{KL}(\rho_{1} \| \rho_{2}) := \int \log \frac{\rmd \rho_{1}}{\rmd \rho_{2}}(x) \rmd\rho_{1}(x)
\end{equation*}
where \(\frac{\rmd \rho_{1}}{\rmd \rho_{2}}\) is the relative density of \(\rho_{1}\) w.r.t. \(\rho_{2}\).
When \(\rho_{1}\) is not absolutely continuous w.r.t. \(\rho_{2}\), \(\mathsf{KL}(\rho_{1} \| \rho_{2}) = \infty\) by convention.
\end{definition}

\begin{definition}[\(\sfD_{q}\)]
Let \(\rho_{1}, \rho_{2}\) be probability measures over \(\bbR^{d}\) such that \(\rho_{1}\) is absolutely continuous w.r.t. \(\rho_{2}\).
For \(q > 1\), the \(q\)-R\'{e}nyi divergence between \(\rho_{1}\) and \(\rho_{2}\) is defined as
\begin{equation*}
    \sfD_{q}(\rho_{1} \| \rho_{2}) := \frac{1}{q - 1}\log \int \left(\frac{\rmd \rho_{1}}{\rmd \rho_{2}}(x)\right)^{q - 1}  \rmd \rho_{1}(x)~.
\end{equation*}
When \(\rho_{1}\) is not absolutely continuous w.r.t. \(\rho_{2}\), \(\sfD_{q}(\rho_{1} \| \rho_{2}) = \infty\) by convention.
\end{definition}

We now define the Ollivier-Ricci curvature.
Since our setup is specific to the Euclidean case, we state a version of \citet[Def. 1]{ollivier2009ricci} for this setting.
The original definition was given for \(p = 1\), while in this work we use \(p \in \{1, 2\}\).

\begin{definition}
Let \(y_{1}, y_{2} \in \bbR^{d}\) such that \(y_{1} \neq y_{2}\), and \(\bfP = \{P_{X | Y = y}\}_{y \in \bbR^{d}}\) be a collection of measures.
The Ollivier-Ricci curvature \(\kappa_{p}(y_{1}, y_{2})\) of \(\bfP\) along \(y_{1}, y_{2}\) is defined as
\begin{equation*}
    \kappa_{p}(y_{1}, y_{2}) := 1 - \frac{\mathsf{W}_{p}\left(P_{X | Y = y_{1}}, P_{X | Y = y_{2}}\right)}{\|y_{1} - y_{2}\|}~.
\end{equation*}
\end{definition}

%% file: two-scale-results.tex
\section{Two-scale criteria via \(\Phi\)-entropies}
\label{sec:two-scale-criteria}

\subsection{Problem setup}
\label{sec:prob-setup}

We consider the following setup in this work, expanding on \cref{sec:intro-two-scale}.
Let \(\rho\) be a probability measure over \(\bbR^{d}\) with density \(\rmd\rho\).
Each component \(P_{X | Y = y}\) and is assumed to be in \(\calP_{\mathrm{ac}}(\bbR^{d})\) with support is independent of \(y\).
\begin{equation}
\label{eq:kernel-general}
    \rmd P_{X | Y = y}(x) = \frac{1}{Z(y)} \cdot \exp\left(-G(x, y)\right)\rmd x~; \qquad p_{X | Y = y}(x) := \frac{\rmd P_{X | Y = y}(x)}{\rmd x}~.
\end{equation}
Above, \(Z(y)\) is the normalisation constant given by \(Z(y) = \int \exp(-G(x, y)) \rmd x\).
Recall that the joint distribution \(\nu\) over \(\bbR^{d} \times \bbR^{d}\) based on \(\rho\) and \(\bfP\) has density 
\begin{equation}
\label{eq:joint-density}
    \rmd \nu(x, y) = \rmd P_{X | Y = y}(x) \rmd \rho(y)~.
\end{equation}
Marginalising over the second variable \(y\) results in a mixture distribution with mixture components given by \(\bfP\) and mixing measure \(\rho\).
The density of \(\mu\) is
\begin{equation}
\label{eq:mixture-density}
    \rmd\mu(x) = \int_{\bbR^{d}} \rmd P_{X | Y = y}(x) \rmd\rho(y)~.
\end{equation}
We assume necessary regularity conditions for the identity \(\bbE_{x \sim P_{X | Y = y}}[\nabla_{y}\log p_{X | Y = y}(x)] = 0\) to hold for all \(y \in \bbR^{d}\); in particular for moving the gradient operator \(\nabla_{y}\) outside the expectation.

\subsection{Main results}

First, we define certain conditions that complete the two-scale criteria considered in this work.

\begin{description}
    \item [Exchange criterion (abbrev. \textsf{Exc}): ]
    \begin{equation*}
        \forall f \in \cssf{\bbR^{d}}, y \in \bbR^{d}~, \quad \left\|\bbE_{x \sim P_{X | Y = y}}\left[\nabla_{y} \log p_{X | Y = y}(x) \cdot f(x)\right]\right\|^{2} \leq \bar{L}^{2} \cdot \bbE_{x \sim P_{X | Y = y}}[\|\nabla f(x)\|^{2}]~.\tag{\textsf{Exc}}\label{eq:2scale-exc}
    \end{equation*}

    \item [Variance criterion (abbrev. \textsf{Var}): ]
    \begin{equation*}
        \forall~u, y \in \bbR^{d}~, \quad\bbE_{x \sim P_{X | Y = y}}\left[\left\langle u, \nabla_{y}\log p_{X | Y = y}(x)\right\rangle^{2}\right] \leq \bar{L}^{2} \cdot \|u\|^{2}~.\tag{\textsf{Var}}\label{eq:2scale-bvar}
    \end{equation*}
    \item [Moment generating function criterion (abbrev. \textsf{MGF}): ]
    \begin{equation*}
        \forall~u, y \in \bbR^{d}~, \quad \log \bbE_{x \sim P_{X | Y = y}}\left[\exp\left(\left\langle u, \nabla_{y}\log p_{X | Y = y}(x)\right\rangle \right)\right] \leq \frac{\bar{L}^{2} \cdot \|u\|^{2}}{2}~.\tag{\textsf{MGF}}\label{eq:2scale-bmgf}
    \end{equation*}
\end{description}

\subsubsection{Some remarks}
By definition, verifying \ref{eq:2scale-exc} is not straightforward, and is in contrast to \ref{eq:2scale-bvar} and \ref{eq:2scale-bmgf} which solely involve \(\nabla_{y}\log p_{X | Y = y}\) and is easier to verify.
However, as we demonstrate in the results presented in this section, \ref{eq:2scale-exc} leads to a more general assertion for \(\Phi\)-Sobolev inequalities.
We also note that \(\bar{L}\)-\textsf{MGF} implies \(\bar{L}\)-\textsf{Var}.
In particular, by exponentiating both sides in \ref{eq:2scale-bmgf} and taking \(u \leftarrow \lambda u\) for any \(u\) and \(\lambda > 0\), one has by the MacLaurin series for \(e^{x}\) that
\begin{equation*}
    \bbE_{x \sim P_{X | Y = y}}\left[1 + \lambda \cdot \langle u, \nabla_{y} \log p_{X | Y = y}(x)\rangle + \frac{\lambda^{2}}{2} \cdot \langle u, \nabla_{y} \log p_{X | Y = y}(x)\rangle^{2} + \ldots \right] \leq 1 + \frac{\lambda^{2} \cdot \bar{L}^{2} \cdot \|u\|^{2}}{2} + \ldots
\end{equation*}
Divide both sides by \(\lambda^{2}\), and since \(\bbE_{x \sim P_{X | Y = y}}[\nabla_{y}\log p_{X | Y = y}(x)] = 0\), taking the limit \(\lambda \to 0\) on both sides results in
\begin{equation*}
    \frac{1}{2} \cdot \bbE_{x \sim P_{X | Y = y}}[\langle u, \nabla_{y} \log p_{X | Y = y}(x)\rangle^{2}] \leq \frac{\bar{L}^{2} \cdot \|u\|^{2}}{2}
\end{equation*}
which is precisely \ref{eq:2scale-bvar}.
The conditions \ref{eq:2scale-exc}, \ref{eq:2scale-bvar} and \ref{eq:2scale-bmgf} can be viewed as complementary to the criteria in \citet{chen2021dimension}, who (in informal terms) assume that the family \(\bfP\) is ``confined'' i.e., for all pairs \(y, y' \in \calY\), \(P_{X | Y = y}\) and \(P_{X | Y = y'}\) are not too different as encoded by bounds on the worst case pairwise discrepancy between elements in \(\bfP\).
Their results also only concern the mixture \(\mu\).

The following theorems give results for the joint and marginal distributions \(\nu\) and \(\mu\) respectively, answering \textbf{Q1} and \textbf{Q2} posed in \cref{sec:intro-two-scale}.

\subsubsection{Main theorems}

\begin{theorem}
\label{thm:Phi-general-joint}
Consider the setup in \cref{sec:prob-setup}.
Let \(\Phi : \calS \to \bbR\) be a twice differentiable convex function.
Assume that \(\rho\) satisfies \(\PhiSI{\alpha}\) and that \(P_{y}\) satisfies \(\PhiSI{\beta}\) for all \(y\).
For \(\bar{L} > 0\), define
\begin{align*}
    \zeta(\alpha, \beta, \bar{L}) &= \inf_{\sfC > 0} \max\left\{\beta + \alpha \cdot (1 + \sfC^{-1}) \cdot \bar{L}^{2},~\alpha \cdot (1 + \sfC)\right\} \\
    &= \frac{1}{2}\left(\alpha + \beta + \alpha \cdot \bar{L}^{2} + \sqrt{4\alpha^{2} \cdot \bar{L}^{2} + (\beta - \alpha + \alpha \cdot \bar{L}^{2})^{2}}\right)~.
\end{align*}
The following statements hold.
\begin{enumerate}[leftmargin=*]
    \item If \(\Phi\) is such that \(\frac{1}{\Phi''}\) is concave and \ref{eq:2scale-exc} holds with \(\bar{L}\), then \(\nu\) satisfies \(\PhiSI{\zeta(\alpha, \beta, \bar{L})}\).
    \item If \(\Phi(t) = t^{2}\) with \(\calS = \bbR\) and \ref{eq:2scale-bvar} holds with \(\bar{L}\), then \(\nu\) satisfies \(\PI{\zeta(\alpha, \beta, \bar{L} \cdot \sqrt{\beta})}\).
    \item If \(\Phi(t) = t\log(t)\) with \(\calS = (0, \infty)\) and \ref{eq:2scale-bmgf} holds with \(\bar{L}\), then \(\nu\) satisfies \(\LSI{\zeta(\alpha, \beta, \bar{L} \cdot \sqrt{\beta})}\).
\end{enumerate}
\end{theorem}

We also state a theorem for the mixture distribution \(\mu\) (\cref{eq:mixture-density}), addressing \textbf{Q2}.

\begin{theorem}
\label{thm:Phi-general-mixture}
Consider the setup in \cref{sec:prob-setup}.
Let \(\Phi : \calS \to \bbR\) be a twice differentiable convex function.
Assume that \(\rho\) satisfies \(\PhiSI{\alpha}\) and that \(P_{y}\) satisfies \(\PhiSI{\beta}\) for all \(y\).
For \(\bar{L} > 0\), define
\begin{equation*}
    \xi(\alpha, \beta, \bar{L}) = \beta + \alpha \cdot \bar{L}^{2}~.
\end{equation*}
The following statements hold.
\begin{enumerate}[leftmargin=*]
    \item If \(\Phi\) is such that \(\frac{1}{\Phi''}\) is concave and \ref{eq:2scale-exc} holds with \(\bar{L}\), then \(\mu\) satisfies \(\PhiSI{\xi(\alpha, \beta, \bar{L})}\).
    \item If \(\Phi(t) = t^{2}\) with \(\calS = \bbR\) and \ref{eq:2scale-bvar} holds with \(\bar{L}\), then \(\mu\) satisfies \(\PI{\xi(\alpha, \beta, \bar{L} \cdot \sqrt{\beta})}\).
    \item If \(\Phi(t) = t\log(t)\) with \(\calS = (0, \infty)\) and \ref{eq:2scale-bmgf} holds with \(\bar{L}\), then \(\mu\) satisfies \(\LSI{\xi(\alpha, \beta, \bar{L} \cdot \sqrt{\beta})}\).
\end{enumerate}
\end{theorem}

The proofs of \cref{thm:Phi-general-joint,thm:Phi-general-mixture} are given in \cref{sec:prf:joint,sec:prf:mixture} respectively.
In doing so, we also identify the following hierarchy between \ref{eq:2scale-exc}, \ref{eq:2scale-bvar} and \ref{eq:2scale-bmgf}.

\begin{figure}[H]
\begin{tikzpicture}[
  node distance=1.5cm,
  box/.style={draw, rectangle, minimum height=0.5cm, minimum width=3.5cm, align=center, font=\small}
]

\node[box] (node1) {$\mathbf{P}$ satisfies $\mathsf{LSI}(\beta)$ + $\mathsf{MGF}(\bar{L})$};
\node[box, right=of node1] (node2) {$\mathbf{P}$ satisfies $\mathsf{PI}(\beta)$ + $\mathsf{Var}(\bar{L})$};
\node[box, right=of node2] (node3) {$\mathbf{P}$ satisfies $\mathsf{Exc}(\sqrt{\beta} \cdot \bar{L})$};

\draw[-{Stealth}, thick] (node1) -- (node2);
\draw[-{Stealth}, thick] (node2) -- (node3);

\draw[-{Stealth}, thick, bend left=20] (node1) to (node3);

\end{tikzpicture}
\caption{Relation between \ref{eq:2scale-exc}, \ref{eq:2scale-bvar}, and \ref{eq:2scale-bmgf}}
\label{fig:hierarchy}
\end{figure}

\subsubsection{Special cases when \(\bfP\) is not indexed by \(y\)}

In deriving \cref{thm:Phi-general-joint,thm:Phi-general-mixture}, we obtain a result for the setting where \(P \in \bfP\) doesn't depend measurably on \(y\).
In this case, \(\nu\) is a product distribution.
Since \(\mu = \rho\), we instead look at the convolution of \(\rho\) and \(P\).

\begin{lemma}
\label{lem:conv-product-measure}
Let \(\Phi\) be a function such that \(\frac{1}{\Phi''}\) is concave.
If \(\rho\) and \(P\) satisfy \(\PhiSI{\alpha}\) and \(\PhiSI{\beta}\) respectively, then
\begin{itemize}[leftmargin=*,itemsep=0pt]
    \item the product distribution \(\rho \otimes P\) satisfies \(\PhiSI{\max\{\alpha, \beta\}}\), and
    \item the convolved distribution \(\rho \ast P\) satisfies \(\PhiSI{\alpha + \beta}\).
\end{itemize}
\end{lemma}

For \(\rho \otimes P\), the above lemma implies the result of \cref{thm:Phi-general-joint} with \(\bar{L} = 0\).
For \(\rho \ast P\), recall that \(z \sim \rho \ast P\) is equal in distribution to \(x + y\) where \(x \sim P\) and \(y \sim \rho\).
The first part of \cref{lem:conv-product-measure} applied to test functions of the form \(\psi(x, y) \leftarrow f(x + y)\) for \(f \in \cssf{\bbR^{d}}\), one can infer the weaker result: \(\rho \ast P\) satisfies \(\PhiSI{2 \cdot \max\{\alpha, \beta\}}\), and the second result of \cref{lem:conv-product-measure} improves this to \(\PhiSI{\alpha + \beta}\), which is strictly better.
See \cref{sec:prf:conv-product-measure} for the proof.

With regards to the tightness of the constant, the constant implied by the first part of \cref{lem:conv-product-measure} is tight and is realised in the case where \(\rho = \calN(0, \sigma^{2}_{1})\) and \(P = \calN(0, \sigma_{2}^{2})\).
From \citet[Cor. 9]{chafai2004entropies}, for \(\Phi\) satisfying the assumption in \cref{lem:conv-product-measure}, we know \(\rho, P\) satisfy \(\PhiSI{\sigma_{1}^{-2}}, \PhiSI{\sigma_{2}^{-2}}\) respectively.
Specifically, \(\rho \otimes P\) is a bivariate Gaussian \(\calN(0, \mathrm{diag}([\sigma_{1}^{2}, \sigma_{2}^{2}]))\) and satisfies \(\PhiSI{\min\{\sigma_{1}^{2}, \sigma_{2}^{2}\}^{-1}}\).
The constant \(\min\{\sigma_{1}^{2}, \sigma_{2}^{2}\}^{-1}\) is precisely \(\max\{\sigma_{1}^{-2}, \sigma_{2}^{-2}\}\) that is implied by the first part of \cref{lem:conv-product-measure}.

\subsubsection{Sufficient conditions for the two-scale criteria (\ref{eq:2scale-bvar}, \ref{eq:2scale-bmgf})}

Here, we discuss simpler conditions under which the two-scale criteria introduced can be found to hold.
We refer the reader to \cref{app:sec:defs} for the definitions of Lipschitz continuity and sub-Gaussianity.
We also make note of the following identity.
\begin{equation*}
    \nabla_{y}\log p_{X | Y = y}(x) = \bbE_{x \sim P_{X | Y = y}}[\nabla_{2}G(x, y)] - \nabla_{2}G(x, y)~.
\end{equation*}

\begin{condlist}[leftmargin=*]
    \item \underline{Assume that \(x \mapsto \nabla_{y}\log p_{X | Y = y}(x)\) is \(B\)-bounded for all \(y\).}
    \\
    In this setting, for all \(x, y \in \bbR^{d}\), we have that \(\|\nabla_{y}\log p_{X | Y = y}(x)\| \leq B\).
    \citet{mou2019sampling} consider this condition in their Lemma 1 towards answering \textbf{Q1}.
    Under this condition, we have for any \(u, y \in \bbR^{d}\) that
    \begin{align*}
        \bbE_{x \sim P_{X | Y = y}}\left[\left\langle u, \nabla_{y}\log p_{X | Y = y}(x)\right\rangle^{2}\right] &\leq \bbE_{x \sim P_{X | Y = y}}\left[\|u\|^{2} \cdot \|\nabla_{y}\log p_{X | Y = y}(x)\|^{2}\right] \\
        &\le B^{2} \cdot \|u\|^{2}~,
    \end{align*}
    which shows that \ref{eq:2scale-bvar} holds with \(\bar{L} = B\).
    Moreover, for any \(u \in \bbR^{d}\), this implies that \(x \mapsto \langle u, \nabla_{y}\log p_{X | Y = y}(x)\rangle\) lies in the range \([-\|u\| \cdot B, \|u\| \cdot B]\) by the Cauchy-Schwarz inequality.
    By Hoeffding's lemma, we have that for any \(u, y \in \bbR^{d}\) that
    \begin{equation*}
        \log \bbE_{x \sim P_{X | Y = y}}\left[\exp\left(\left\langle u, \nabla_{y}\log p_{X | Y = y}(x)\right\rangle \right)\right] \leq \frac{B^{2} \cdot \|u\|^{2}}{2}~.
    \end{equation*}
    This shows that \ref{eq:2scale-bmgf} holds with \(\bar{L} = B\).

    \item \underline{Assume that \(x \mapsto \nabla_{y}\log p_{X | Y = y}(x)\) has \(B\)-bounded variance w.r.t. \(P_{X | Y = y}\) for all \(y\).}
    \\
    This states that for all \(y \in \bbR^{d}\), \(\bbE_{x \sim P_{X | Y = y}}[\|\nabla_{y}\log p_{X | Y = y}(x)\|^{2}] \leq B^{2}\).
    \citet{ge2020simulated} consider this assumption in their Theorem D.3 towards answering \textbf{Q1}.
    In the context of information theory, the quantity \(\bbE_{x \sim P_{X | Y = y}}[\|\nabla_{y}\log p_{X | Y =y}(x)\|^{2}]\) is referred to as the \emph{pointwise statistical Fisher information} \citep[\S III.B]{wibisono2017information}.
    From the workings above, we can also infer that \ref{eq:2scale-bvar} holds with \(\bar{L} = B\) under this condition.

    \item \underline{Assume that the random variable \(\nabla_{y}\log p_{X | Y = y}(x)\) is \(\sigma\)-sub-Gaussian for \(x \!\sim P_{X | Y = y}\) for all \(y\).}\\
    In other words, for every \(y\) and unit vector \(u \in \bbR^{d}\), \(\langle u, \nabla_{y} \log p_{X | Y = y}(x, y)\rangle\) is sub-Gaussian w.r.t. \(P_{X | Y =y}\) with variance proxy \(\sigma^{2}\).
    This can be equivalently written as: for all \(y \in \bbR^{d}\) and \(u \in \bbR^{d}\),
    \begin{equation*}
        \log \bbE_{x \sim P_{X | Y = y}}\left[\exp\left(\left\langle u, \nabla_{y}\log p_{X | Y = y}(x)\right\rangle \right)\right] \leq \frac{\sigma^{2} \cdot \|u\|^{2}}{2}~,
    \end{equation*}
    which is precisely \ref{eq:2scale-bmgf} with \(\bar{L} = \sigma\).

    \item \label{cond:lips-cond} \underline{Assume that \(x \mapsto \nabla_{2}G(x, y)\) is \(L\)-Lipsschitiz continuous and \(P_{X | Y= y}\) for all \(y\).}\\
    This condition is the focus of \citet[Thm. 2, Lem. 7]{otto2007new}.
    Under this condition, for any \(u \in \bbR^{d}\), we have that \(x \mapsto \langle u, \nabla_{2}G(x, y)\rangle\) is \(L \cdot \|u\|\)-Lipschitz continuous.
    When \(\bfP\) satisfies \(\PI{\beta}\), we get
    \begin{align*}
        \bbE_{x \sim P_{X | Y = y}}\left[\left\langle u, \nabla_{y}\log p_{X | Y = y}(x)\right\rangle^{2}\right] &=
        \bbE_{x \sim P_{X | Y = y}}\left[\left\langle u, \nabla_{2}G(x, y) - \bbE_{x \sim P_{X | Y = y}}[\nabla_{2}G(x, y)]\right\rangle^{2}\right] \\
        &= \bbV_{x \sim P_{X | Y = y}}[\left\langle u, \nabla_{2}G(x, y)\right\rangle] \\
        &\leq \beta \cdot L^{2} \cdot \|u\|^{2}~.
    \end{align*}
    Therefore, \ref{eq:2scale-bvar} holds with \(\bar{L} = \sqrt{\beta} \cdot L\).
    Also, by Herbst's argument \citep[Prop. 5.41]{bakry2014analysis} for when \(\bfP\) satisfies \(\LSI{\beta}\), we have using the Lipschitz continuity of \(x \mapsto \langle u, \nabla_{2}G(x, y)\rangle\) that
    \begin{align*}
        \log \bbE_{x \sim P_{X | Y = y}}\left[\exp\left(\left\langle u, \nabla_{y}\right.\right.\right.\!\!\!\! &\left.\left.\left.\log p_{X | Y = y}(x)\right\rangle \right)\right] \\ 
        &= \log \bbE_{x \sim P_{X | Y= y}}\left[\exp\left(\langle u, \nabla_{2}G(x, y)\rangle - \bbE_{x \sim P_{X | Y =y}}\left[\langle u, \nabla_{2}G(x, y)\rangle\right]\right)\right] \\
        &\leq \frac{\beta \cdot L^{2} \cdot \|u\|^{2}}{2}~.
    \end{align*}
    This proves that \ref{eq:2scale-bmgf} holds with \(\bar{L} = \sqrt{\beta} \cdot L\).
\end{condlist}

\subsection{Examples}

Here, we discuss examples of Markov chains that satisfy these conditions.
We assume that these Markov chains are targeted at sampling from \(\pi^{\star}\) with density \(\rmd\pi^{\star}(x) = \exp(-V^{\star}(x))\rmd x\).

\subsubsection{The Unadjusted Langevin Algorithm}
When \(V^{\star}\) is differentiable, the unadjusted Langevin algorithm (ULA) \citep{roberts1996exponential} is defined by the iteration
\begin{equation*}
    X^{(k + 1)} = X^{(k)} - \eta \cdot \nabla V^{\star}(X^{(k)}) + \sqrt{2\eta} \cdot \xi_{k}~; \quad \xi_{k} \sim \calN(0, \rmI_{d})~.
\end{equation*}
The parameter \(\eta\) is the step size and the algorithm can be viewed as the Euler-Maruyama discretisation of the continuous-time overdamped Langevin dynamics defined by the SDE
\begin{equation*}
    \rmd X_{t} = -\nabla V^{\star}(X_{t})\rmd t + \sqrt{2}\rmd B_{t}
\end{equation*}
where \((B_{t})_{t \geq 0}\) is the Brownian motion.
Notably, the stationary distribution for this dynamics is \(\pi^{\star}\), which makes it a viable dynamics to discretise.
Thus, ULA defines the Markov kernel \(\bfP^{\text{ULA}} = \{P^{\text{ULA}}_{y}\}_{y \in \bbR^{d}}\) where \(P^{\text{ULA}}_{y} = \calN(y - \eta \cdot \nabla V^{\star}(y), 2h \cdot \rmI_{d})\).
This kernel satisfies \(\PhiSI{\eta}\) \citep[Cor 2.1]{chafai2004entropies} for any \(\Phi\) such that \(\frac{1}{\Phi''}\) is concave.
Next, we show that \(\bfP^{\text{ULA}}\) satisfies \ref{eq:2scale-exc}.
By definition, we have
\begin{equation*}
    \nabla_{y}\log p_{X | Y = y}^{\text{ULA}}(x) = \frac{1}{\eta} \left[(\rmI_{d} - \eta \cdot \nabla^{2}V^{\star}(y))\right](x - (y - \eta \cdot \nabla V^{\star}(y))~.
\end{equation*}
Consequently,
\begin{align*}
    \bbE_{x \sim P_{X | Y = y}^{\text{ULA}}}\left[f(x) \cdot \nabla_{y} \log p_{X | Y = y}^{\text{ULA}}(x) \right] &= [\rmI_{d} - \eta \cdot \nabla^{2}V^{\star}(y)] \int f(x) \left(- \nabla_{x} p_{X | Y = y}^{\text{ULA}}(x)\right) \rmd x \\
    &= [\rmI_{d} - \eta \cdot \nabla^{2}V^{\star}(y)] \cdot \int \nabla f(x) \cdot p_{X | Y = y}^{\text{ULA}}(x) \rmd x
\end{align*}
where the last step uses integration by parts (\cref{prop:ibp}).
Therefore,
\begin{align*}
    \left\|\bbE_{x \sim P_{X | Y = y}^{\text{ULA}}}[\nabla_{y}\log p_{X | Y = y}(x) \cdot f(x)]\right\|^{2} &\leq \sup_{y} \left\|\rmI_{d} - \eta \cdot \nabla^{2}V^{\star}(y)\right\|^{2} \cdot \left\|\bbE_{x \sim P_{X | Y = y}^{\text{ULA}}}[\nabla f(x)]\right\|^{2} \\
    &\leq \sup_{y} \left\|\rmI_{d} - \eta \cdot \nabla^{2}V^{\star}(y)\right\|_{\op}^{2} \cdot \bbE_{x \sim P_{X | Y = y}^{\text{ULA}}}[\|\nabla f(x)\|^{2}]~.
\end{align*}
This establishes that \(\bfP^{\text{ULA}}\) satisfies \ref{eq:2scale-exc} with \(\bar{L} = \sup_{y} \left\|\rmI_{d} - \eta \cdot \nabla^{2}V^{\star}(y)\right\|_{\op}\).

We demonstrate the ease of verifying \ref{eq:2scale-bvar} and \ref{eq:2scale-bmgf}.
Notice that \(x\mapsto \nabla_{y}\log p_{X | Y = y}^{\text{ULA}}(x)\) is Lipschitz continuous for all \(y\); this is evident from the gradient derived above which is linear in \(x\).
In particular, the Lipschitz continuity constant is \(\eta^{-1} \sup_{y} \|\rmI_{d} - \eta \cdot \nabla^{2}V^{\star}(y)\|_{\op}\).
Since \(P_{y}^{\text{ULA}}\) satisfies \(\LSI{\eta}\) and \(\PI{\eta}\), from the sufficient condition \ref{cond:lips-cond} for \ref{eq:2scale-bvar} and \ref{eq:2scale-bmgf} , we have that \(\bfP^{\text{ULA}}\) satisfies \ref{eq:2scale-bvar} and \ref{eq:2scale-bmgf} with \(\bar{L} = \eta^{-\nicefrac{1}{2}} \cdot \sup_{y} \|\rmI_{d} - \eta \cdot \nabla^{2}V^{\star}(y)\|_{\op}\).
This is consistent with verifying that \(\bfP^{\text{ULA}}\) satisfies \ref{eq:2scale-exc} through the hierarchy in \cref{fig:hierarchy}.

\paragraph{When \(V^{\star}\) is \(\alpha\)-strongly convex and \(L\)-smooth}

We highlight this case as a setting where it is possible to quantify \(\sup_{y} \|\rmI_{d} - \eta \cdot \nabla^{2} V^{\star}(y)\|_{\op}\) precisely.
\(L\)-smoothness of \(V^{\star}\) is equivalent to \(\nabla V^{\star}\) being \(L\)-Lipschitz continuous.
In this setting, we have
\begin{equation*}
    \sup_{y} \|\rmI_{d} - \eta \cdot \nabla^{2}V^{\star}(y)\|_{\op} = \max\{|1 - \eta \cdot L|, |1 - \eta \cdot m|\}~.
\end{equation*}

\subsubsection{The Proximal Sampling Algorithm}

The proximal sampling algorithm of \citet{lee2021structured} provides an alternative approach to sampling from \(\pi^{\star}\) by constructing an augmented distribution \(\tilde{\pi}^{\star}\) with parameter \(\eta > 0\) referred to as the step size, and whose density is given by
\begin{equation*}
    \rmd\tilde{\pi}^{\star}(x, y) \propto \exp\left(-\left\{\potential{}^{\star}(x) + \frac{\|y - x\|^{2}}{2\eta}\right\}\right)\rmd x\rmd y~,
\end{equation*}
and the procedure performs Gibbs' sampling over this augmented distribution as shown.
\begin{equation*}
    Y^{(k + 1)} \mid X^{(k)} \sim \tilde{\pi}^{\star}(y | x = X^{(k)})~; \qquad X^{(k + 1)} \mid Y^{(k + 1)} \sim \tilde{\pi}^{\star}(x | y = Y^{(k + 1)})~.
\end{equation*}
The conditional distribution \(\tilde{\pi}^{\star}(y | x = X^{(k)}\) (referred to as the forward kernel) is a Gaussian distribution with mean \(X^{(k)}\) and covariance \(2\eta \rmI_{d}\).
The forward kernel \(\bfP^{\text{Prox},F}_{y}\) where \(P_{y}^{\text{Prox, F}} = \calN(y, 2\eta \rmI_{d})\) satisfies \ref{eq:2scale-exc} with \(\bar{L} = 1\), and satisfies \(\PhiSI{\eta}\) for \(\Phi\) such that \(\frac{1}{\Phi''}\) is concave.

On the other hand, the conditional \(\tilde{\pi}^{\star}(x | y= Y^{(k + 1)})\) (referred to as the backward kernel) is non-trivial as it involves \(V^{\star}\).
We show that \(\bfP^{\text{Prox},B}_{y}\) where \(P_{y}^{\text{Prox,B}} = \widetilde{\pi}^{\star}(. | y = y)\) satisfies \ref{eq:2scale-exc} as well.
By writing the density of \(P_{y}^{\text{Prox,B}}\)
\begin{equation*}
    p_{X | Y = y}^{\text{Prox,B}}(x) = \frac{1}{Z(y)} \cdot \exp\left(-V^{\star}(x) - \frac{1}{2\eta}\|y - x\|^{2}\right)\rmd x
\end{equation*}
it can be immediately seen that by taking \(G^{\text{Prox,B}}(x, y) = V^{\star} + \frac{1}{2\eta}\|y- x\|^{2}\), we have
\(x\mapsto \nabla_{2}G^{\text{Prox,B}}(x, y) = \frac{x - y}{\eta}\), which is \(\frac{1}{\eta}\)-Lipschitz continuous.
Hence by \ref{cond:lips-cond}, if \(P^{\text{Prox,B}}_{y}\) satisfies \(\mathsf{PI}\) or \(\mathsf{LSI}\) for all \(y \in \bbR^{d}\), we can establish that \(\bfP^{\text{Prox,B}}\) satisfies \ref{eq:2scale-bvar} or \ref{eq:2scale-bmgf} respectively.
Below, we highlight cases where such functional inequalities for \(\bfP^{\text{Prox,B}}\) provably holds.

\paragraph{When \(\potential{}^{\star}\) is \(m\)-strongly convex.}
In this case, \(x \mapsto G^{\mathrm{Prox,B}}(x, y)\) is \(\left\{\mu + \frac{1}{\eta}\right\}\)-strongly convex for every \(y\).
By the Bakry-Emery criterion, for any \(y \in \bbR^{d}\)
\begin{equation*}
    P^{\mathrm{Prox,B}}_{X | Y = y} \text{ satisfies } \LSI{\beta}; \quad \beta = \left(m + \frac{1}{\eta}\right)^{-1}~.
\end{equation*}

We also have two cases where \(V^{\star}\) are perturbations of a \(m\)-strongly convex part, which means that \(V^{\star}\) is not necessarily convex.

\paragraph{When \(\potential{}^{\star} = \potential{}^{\star}_{m-\text{s.c.}} + \potential{}^{\star}_{B-\text{bdd.}}\)}

Here, \(\potential{}^{\star}_{B-\text{bdd.}}\) is a function that satisfies \(\mathrm{osc}(\potential{}^{\star}_{B-\text{bdd.}}) = B\) where \(\mathrm{osc}(f) := \sup_{x} f(x) - \inf_{x} f(x)\).
From the popular Holley-Stroock perturbation result \citep[Prop. 5.1.6]{bakry2014analysis}, we have for any \(y \in \bbR^{d}\) that
\begin{equation*}
    P^{\mathrm{Prox,B}}_{X | Y = y} \text{ satisfies } \LSI{\beta}; \quad \beta = e^{B} \cdot \left(m + \frac{1}{\eta}\right)^{-1}~.
\end{equation*}

\paragraph{When \(\potential{}^{\star} = \potential{}^{\star}_{m-\text{s.c.}} + \potential{}^{\star}_{L-\text{Lip}}\).}
Here, \(\potential{}^{\star}_{\text{Lip}}\) is \(L\)-Lipschitz continuous function
For every \(y\), \(x \mapsto G^{\mathrm{Prox,B}}(x, y)\) is sum of a \(\left(m + \frac{1}{\eta}\right)\)-strongly convex function and a \(L\)-Lipschitz continuous function.
The following result shown by \citet{brigati2024heat} pertains to this specific setting.
\begin{proposition}[{\citep[Thm. 1.4]{brigati2024heat}}]
    Let \(\pi\) be a distribution over \(\bbR^{d}\) with density of the form \(\rmd\pi(x) \propto e^{-V(x)}\rmd x\), where \(V = V_{m-\text{s.c.}} + V_{L-\text{Lip}}\).
    Then, there exists a \(\frakL\)-Lipschitz continuous map \(\frakT\) such that \(\frakT_{\#}\calN(0, \rmI_{d}) = \pi\) where
    \begin{equation*}
        \frakL = \frac{1}{\sqrt{m}} \cdot \exp\left(\frac{L^{2}}{2m} + \frac{2L}{\sqrt{m}}\right)~.
    \end{equation*}
    Consequently, \(\pi\) satisfies \(\LSI{\frakL^{2}}\).
\end{proposition}
With the above proposition, we have for any \(y\) that
\begin{equation*}
    P^{\mathrm{Prox,B}}_{X | Y = y} \text{ satisfies } \LSI{\beta}; \quad \beta = \frac{\eta}{\eta \cdot \mu + 1} \cdot \exp\left(\frac{\eta \cdot L^{2}}{\eta \cdot \mu + 1} + \frac{4L \cdot \sqrt{\eta}}{\sqrt{\eta \cdot \mu + 1}}\right)~.
\end{equation*}

We conclude the discussion for the proximal sampling algorithm by reiterating (again) through this example the ease of checking \ref{eq:2scale-bvar} and \ref{eq:2scale-bmgf}.

%% file: mcmc-applications.tex
\section{Applications to MCMC}
\label{sec:applications}

In this section, we highlight how the two-scale criteria help us understand components of errors in MCMC estimation as depicted in \cref{eq:error-decomp}.
In summary, the goal is to estimate \(\bbE_{\pi^{\star}}[f]\) using a collection of \(N\) samples \(\calS_{N}\) using the estimate \(\widehat{F}(f, \calS_{N})\) defined as the empirical average of evaluations of \(f\) on \(\calS_{N}\).
Here, we specifically focus on the setting where \(f\) is Lipschitz continuous.
We also target simple procedures for collecting \(\calS_{N}\).
\begin{description}
    \item [Scheme 1]: Run \(N\) Markov chains in parallel for \(J\) transitions, and accumulate \(\calS_{N}\) as the final iterates of these \(N\) chains.
    \item [Scheme 2]: Run \(1\) Markov chain for \(J\) transitions, and collect a sample every \(K\) transitions thereafter.
\end{description}
The parameter \(J\) is referred to as the \emph{burnin period}.
While the first procedure is beneficial in that it collects independent samples, it discards all previous iterates generated from the Markov chains which might still be valuable.
The second procedure addresses this drawback at the cost of producing dependent samples.
While the dependency cannot be eliminated, the correlation between samples can reduced via the \emph{thinning} parameter \(K\).

\subsection{From \textsf{Exc} to mixing}

We first discuss how \ref{eq:2scale-exc} can be used to infer of mixing in \(\sfW_{1}\) for \(\bfP\) to its stationary distribution.
This is particularly valuable since \(f\) is \(L\)-Lipschitz continuous, and how this yields bounds for the convergence term in the error decomposition in \cref{eq:error-decomp}.
In particular, we focus on the Ollivier-Ricci curvature of \(\bfP\) and its implications.
The developments of this subsection are facilitated by the dual form of \(\sfW_{1}\) and \(\sfW_{2}\) distances (called the Kantorovich-Rubenstein duality \citep[Thm. 5.10]{villani2009optimal}).
This is central to the discussion of this section.
\begin{proposition}
\label{prop:W1-dual}
    Let \(\rho_{1}, \rho_{2}\) be probability measures over \(\bbR^{d}\).
    The \(1\)-Wasserstein distance between \(\rho_{1}\) and \(\rho_{2}\) is also defined as
    \begin{equation*}
        \sfW_{1}(\rho_{1}, \rho_{2}) = \sup_{g} \bbE_{\rho_{1}}[g] - \bbE_{\rho_{2}}[g]
    \end{equation*}
    where the supremum is taken over all \(1\)-Lipschitz continuous functions.
\end{proposition}
\begin{proposition}
\label{prop:W2-dual}
    Let \(\rho_{1}, \rho_{2}\) be probability measures over \(\bbR^{d}\).
    The \(2\)-Wasserstein distance between \(\rho_{1}\) and \(\rho_{2}\) is also defined as
    \begin{equation*}
        \sfW_{2}(\rho_{2}, \rho_{2})^{2} = \sup_{g} \bbE_{y \sim \rho_{1}}\left[\left\{\inf_{x} g(y) + \|x - y\|^{2}\right\}\right] - \bbE_{\rho_{2}}[g]
    \end{equation*}
    where the supremum is taken over all absolutely integrable functions w.r.t. \(\rho_{2}\).
\end{proposition}

We use this duality to show that \ref{eq:2scale-exc} implies bounds on the Ollivier-Ricci curvature of \(\bfP\).
\begin{theorem}
\label{thm:ricci-curvature}
Let \(\bfP = \{P_{y}\}\) satisfy \ref{eq:2scale-exc} with \(\bar{L}\).
Then for any \(y_{1}, y_{2} \in \bbR^{d}\) and \(p \in \{1, 2\}\).
\begin{equation*}
    \kappa_{p}(y_{1}, y_{2}) \geq 1 - \bar{L}~.
\end{equation*}
\end{theorem}
We prove this theorem in \Cref{sec:prf:ricci-curvature}.
This result states that when \(\bar{L} < 1\), the Ricci curvature is positive in \(\bbR^{d}\), which has implications for mixing.
Since \(\sfW_{1} \leq \sfW_{2}\) between any two distributions, the curvature \(\kappa_{2}\) that we have established for \(\sfW_{2}\) is actually stronger, but we nonetheless derive them independently.
From \citet[Prop. 20]{ollivier2009ricci}, we can quantify properties how close mixtures obtained from different mixing distributions are relative to the closeness of their mixing distributions; see the corollary below.
\begin{proposition}
\label{prop:w1-contract}
Let \(\mu_{1}, \mu_{2}\) be mixtures of \(\bfP\) generated with mixing distributions \(\rho_{1}, \rho_{2}\) respectively, where \(\rho_{1}, \rho_{2} \in \calP_{\mathrm{ac}}(\bbR^{d})\).
Assume that \(\bfP\) satisfies \ref{eq:2scale-exc} with \(\bar{L}\).
Then for \(p \in \{1, 2\}\)
\begin{equation*}
    \mathsf{W}_{p}(\mu_{1}, \mu_{2}) \leq \bar{L} \cdot \mathsf{W}_{p}(\rho_{1}, \rho_{2})~.
\end{equation*}
Moreover, if \(\bar{L} < 1\), then there exists a unique \(\rho^{\star} \in \mathcal{P}_{\mathrm{ac}}(\bbR^{d})\) such that
\begin{equation*}
    \rmd \rho^{\star}(x) = \int_{\bbR^{d}} p_{X | Y = y}(x)\rmd \rho^{\star}(y)~.
\end{equation*}
In other words, \(\rho^{\star}\) is the stationary distribution of \(\bfP\).
\end{proposition}

Since our focus is on the case where \(f\) is \(L\)-Lipschitz continuous, we can use the duality of \(\mathsf{W}_{1}\) (\cref{prop:W1-dual}) to provide bounds on the convergence component of the error in \cref{eq:error-decomp}.
We provide the proof of the lemma below in \Cref{sec:prf:conv-error-bound}.

\begin{lemma}
\label{lem:conv-error-bound}
Suppose the kernel \(\bfP\) satisfies \ref{eq:2scale-exc} with \(\bar{L} < 1\).
Then we have the following bounds.
        \begin{align*}
            \textbf{Scheme 1:} & & \quad \left|\bbE[\widehat{F}(f; \calS_{N})] - F^{\star}(f; \rho^{\star})\right| & \leq L \cdot \bar{L}^{J} \cdot \mathsf{W}_{1}(\rho_{0}, \rho^{\star})\\
            \textbf{Scheme 2:} & & \quad \left|\bbE[\widehat{F}(f; \calS_{N})] - F^{\star}(f; \rho^{\star})\right| & \leq \frac{1}{N}\sum_{i=1}^{N} \bar{L}^{J + (i - 1)K} \cdot \mathsf{W}_{1}(\rho_{0}, \rho^{\star})~.
        \end{align*}
\end{lemma}

\subsection{Understanding concentration}

Given our understanding of how the population quantity \(\bbE[\widehat{F}(f; \calS_{N})]\) differs from the bias limit \(F^{\star}(f; \rho^{\star})\), we now focus on how \(\widehat{F}(f; \calS_{N})\) concentrates around the mean --- this specifically the first term in \cref{eq:error-decomp} given by
\begin{equation*}
    \widehat{F}(f, \calS_{N}) - \bbE[\widehat{F}(f, \calS_{N})]
\end{equation*}
where the expectation is taken over the randomness in the sampling procedure.
By a generic Chernoff technique, we can bound this deviation in probability as
\begin{align*}
    \bbP\left(\widehat{F}(f, \calS_{N}) - \bbE[\widehat{F}(f, \calS_{N})] > t\right) &= \bbP\left(\exp\left\{\lambda \cdot \left[\widehat{F}(f, \calS_{N}) - \bbE[\widehat{F}(f, \calS_{N})]\right]\right\} > \exp(\lambda t)\right) \\
    &\leq \frac{1}{\exp(\lambda t)} \cdot \bbE\left[\exp\left(\lambda \cdot \left\{\widehat{F}(f, \calS_{N}) - \bbE\left[\widehat{F}(f, \calS_{N})\right]\right\}\right)\right]~.
\end{align*}
In the following lemmas, we give bounds for the expectation on the right hand side.
We make the following assumption about distributions satisfying \(\Phi\mathsf{SI}\) to obtain quantitative rates.
\begin{assumption}
\label{assump:mgf}
If a distribution \(P\) satisfies \(\PhiSI{c}\), then there exists a function \(\varphi_{\Phi}\) and \(\Lambda \in \bbR \cup \{\infty\}\) such that for a \(1\)-Lipschitz continuous function \(g\)
\begin{equation*}
    \bbE_{\sfx \sim P}\left[\exp\left(\lambda \cdot (g - \bbE_{P}[g])\right)\right] \leq \varphi_{\Phi}(\lambda; c)~\quad \forall ~\lambda < \Lambda~.
\end{equation*}
\end{assumption}
This holds concretely in the following choices for \(\Phi\).
\begin{center}
\renewcommand{\arraystretch}{1.5}
\begin{tabular}{ccccc}
\(\Phi(t)\) & Domain of \(\Phi\) & \(\varphi_{\Phi}\) & \(\Lambda\) & Reference \\
\hline
\(t\log(t)\) & \((0, \infty)\) & \(\exp\left(\frac{c\lambda^{2}}{2}\right)\) & \(\infty\) & \citet[Prop. 5.4.1]{bakry2014analysis} \\
\(t^{\nicefrac{2}{p}}\), \(p \in [1, 2)\) & \((0, \infty)\) & \(\left(1 - \frac{c\lambda^{2}(2 - p)}{4}\right)^{-\frac{2}{2 - p}}\) & \(\frac{2}{\sqrt{c(2 - p)}}\) & \citet[Thm. 1]{latala2007between}
\end{tabular}
\end{center}

\begin{lemma}
\label{lem:scheme1-chernoff}
Let the initial distribution \(\rho_{0}\) satisfies \(\PhiSI{\alpha_{0}}\), and the kernel \(\bfP\) satisfies \ref{eq:2scale-exc} and \(\PhiSI{\beta}\) for all \(P_{y} \in \bfP\).
Suppose \cref{assump:mgf} holds for this \(\Phi\) and that \(\frac{1}{\Phi''}\) is concave.
Define \(\alpha_{J} = \beta \sum_{i=0}^{J- 1}\bar{L}^{2i} + \alpha_{0}\bar{L}^{2J}\).
Then, samples collected using \textbf{Scheme 1} satisfies \(\widehat{F}(f; \calS_{N}) - \bbE[\widehat{F}(f; \calS_{N})] > t\) with probability at most
\begin{align*}
\exp\left(-\frac{Nt^{2}}{2\alpha_{J}}\right)~ &\quad \text{when } \Phi(t) = t\log t \\
\exp\left(-\frac{2N}{2 - p}\left\{\bar{\alpha}_{J}(t; p) - 1 - \log(\bar{\alpha}_{J}(t; p) + 1) + \log(2)\right\}\right)~ &\quad \text{when } \Phi(t) = t^{\nicefrac{2}{p}}, p \in [1, 2) \enskip \\
{\scriptstyle \bar{\alpha}_{J}(t; p) = \sqrt{1 + \frac{t^{2}(2 - p)}{\alpha_{J}L^{2}}}}
\end{align*}
\end{lemma}
The proof of this lemma is given in \Cref{sec:prf:scheme1-chernoff}.

\begin{lemma}
\label{lem:scheme2-chernoff}
Let the initial distribution \(\rho_{0}\) satisfies \(\PhiSI{\alpha_{0}}\), and the kernel \(\bfP\) satisfies \ref{eq:2scale-exc} with \(\bar{L} < 1\) and \(\PhiSI{\beta}\) for all \(P_{y} \in \bfP\).
Suppose \cref{assump:mgf} holds for this \(\Phi\) and that \(\frac{1}{\Phi''}\) is concave.
Then, samples collected using \textbf{Scheme 2} satisfies \(\widehat{F}(f; \calS_{N}) - \bbE[\widehat{F}(f; \calS_{N})] > t\) with probability at most
\begin{align*}
\exp\left(-\frac{N^{2}t^{2}}{2L^{2}\frakQ(\alpha_{0}, \beta, \bar{L}, K, J, N)}\right)~ &\quad \text{when } \Phi(t) = t\log t
\end{align*}
where \(\frakQ(\alpha_{0}, \beta, \bar{L}, K, J, N) = \frac{\beta \cdot N}{1 - \bar{L}^{2}} + \alpha_{0}\bar{L}^{2J}\left(1 + \frac{\bar{L}^{2}}{1 - \bar{L}^{2K}} \right)\).
\end{lemma}

The proof of this lemma in given in \Cref{sec:prf:scheme2-chernoff}.
We also provide the workings for the \(\Phi(t) = t^{\nicefrac{2}{p}}\), which can be numerically solved, but analytically is tedious to compute in closed form.

\subsection{Relating the bias and the limiting distribution \(\rho^{\star}\)}

Finally, we address the bias in \cref{eq:error-decomp}.
Since \(f\) is \(L\)-Lipschitz continuous, this quantity can be bounded using \Cref{prop:W1-dual} as
\begin{equation*}
    \left| F^{\star}(f; \rho^{\star}) - F^{\star}(f; \pi^{\star})\right| \leq L \cdot \mathsf{W}_{1}(\rho^{\star}, \pi^{\star})~.
\end{equation*}
Furthermore, when \(\pi^{\star}\) satisfies a Talagrand \(\mathsf{T}_{1}\) inequality, the \(\mathsf{W}_{1}\) distance between \(\rho^{\star}\) and \(\pi^{\star}\) can be bounded by the \(\mathsf{KL}\) divergence between \(\rho^{\star}\) and \(\pi^{\star}\).
This quantity has been studied extensively for a variety of Markov kernels \(\bfP\).

Deviating from the prior exposition where \(f\) is assumed to be Lipschitz continuous, we discuss this error more generally and connect properties of \(f\) to various measures of discrepancy between \(\rho^{\star}\) and\ \(\pi^{\star}\) like the total variation distance and R\'{e}nyi divergences.
We assume that \(\rho^{\star}\) w.r.t. \(\pi^{\star}\) for this discussion, and therefore \(\mathsf{KL}(\rho^{\star} \| \pi^{\star})\) and \(\sfD_{q}(\rho^{\star} \| \pi^{\star})\) are finite quantities.

\begin{lemma}
\label{lem:tv-bounded}
Let \(f\) be a function such that \(|f(x)| \leq B\) for all \(x \in \bbR^{d}\).
Then,
\begin{equation*}
    |F^{\star}(f; \rho^{\star}) - F^{\star}(f; \pi^{\star})| \leq B \cdot \mathsf{TV}(\rho^{\star}, \pi^{\star})~.
\end{equation*}
\end{lemma}

\begin{lemma}
\label{lem:kl-mgf}
Let \(f\) be a function such that \(\frac{1}{\lambda^{p}}\log \bbE_{\pi^{\star}}[\exp(\lambda \cdot (f - \bbE_{\pi^{\star}}[f]))] < \infty\) for all \(\lambda \in \bbR\) and some \(p > 1\).
Then,
\begin{equation*}
    \left|F^{\star}(f; \rho^{\star}) - F^{\star}(f; \pi^{\star})\right| \leq \left(\mathsf{KL}(\rho^{\star} \| \pi^{\star})\right)^{\frac{p - 1}{p}} \cdot \left(\sup_{\lambda > 0} \frac{1}{\lambda^{p}} \log \bbE_{\pi^{\star}}\left[\exp(\lambda \cdot (f - \bbE_{\pi^{\star}}[f]))\right] \right)^{\frac{1}{p}} \cdot \frac{p}{(p - 1)^{\frac{p - 1}{p}}}~.
\end{equation*}
\end{lemma}

\begin{lemma}
\label{lem:renyi-moment}
Let \(f\) be a function such that \(\bbE_{\pi^{\star}}[|f|^{\frac{q}{q - 1}}] < \infty\) for \(q > 1\).
Then,
\begin{equation*}
    \left|F^{\star}(f; \rho^{\star}) - F^{\star}(f; \pi^{\star})\right| \leq \left(\exp((q - 1) \sfD_{q}(\rho^{\star} \| \pi^{\star})) - 1\right)^{\frac{1}{q}} \cdot \bbE_{\pi^{\star}}[|f|^{\frac{q}{q - 1}}]^{\frac{q - 1}{q}}~.
\end{equation*}
\end{lemma}

We remark that \(\sfD_{q}(\rho^{\star} \| \pi^{\star})\) is non-increasing in \(q\).
This implies that
\begin{equation*}
    2 \cdot \mathsf{TV}(\rho^{\star} \| \pi^{\star})^{2} \leq  \mathsf{KL}(\rho^{\star} \| \pi^{\star}) = \lim_{q' \to 1} \sfD_{q'}(\rho^{\star} \| \pi^{\star}) \leq \sfD_{q_{1}}(\rho^{\star} \| \pi^{\star}) \leq \sfD_{q_{2}}(\rho^{\star} \| \pi^{\star})
\end{equation*}
for \(1 < q_{1} \leq q_{2}\).
In other words, having a bound on \(\sfD_{q}(\rho^{\star} \| \pi^{\star})\) for any \(q > 1\) implies a bound for \(\sfD_{q'}(\rho^{\star} \| \pi^{\star})\) for \(q' \leq q\).
A consequence of this is that one would expect weaker conditions on \(f\) to control the bias when given bounds on \(\sfD_{q}(\rho^{\star} \| \pi^{\star})\) than when given bounds on \(\mathsf{KL}(\rho^{\star} \| \pi^{\star})\).
This intuition can be validated by observing the conditions over \(f\) in \Cref{lem:kl-mgf,lem:renyi-moment}.
In particular, for the random variable \(f(\sfx)\) where \(\sfx \sim \pi^{\star}\), the condition in \Cref{lem:kl-mgf} pertains to the MGF, whereas the conditon in \Cref{lem:renyi-moment} pertains to a certain moment of \(f(\sfx)\) for \(x \sim \pi^{\star}\).
Moreover, as \(q\) decreases, \(\frac{q}{q - 1}\) increases which implies control over higher order moments of \(f(\sfx)\).
We prove the above lemmas in \Cref{sec:prf:bias-bounds}.

%% file: proofs.tex
\section{Proofs}
\label{sec:proofs}

\subsection{Proofs of the results in \cref{sec:two-scale-criteria}}

In this section, we give the proofs for \cref{thm:Phi-general-joint,thm:Phi-general-mixture}.
We begin by stating a collection of smaller propositions that come in handy, and the proofs of these propositions are given in \cref{sec:proofs-prop-proofs}.

\begin{proposition}
\label{prop:expected-value-kernel-test-grad}
Let \(P_{X | Y = y}\) be given by \cref{eq:kernel-general}, and \(\psi : \bbR^{d} \times \bbR^{d} \to \bbR\) be a differentiable in both its arguments.
Assuming that \(M(y) := \bbE_{x \sim P_{X | Y = y}}[\psi(x, y)]\) exists for all \(y\), it holds that
\begin{equation*}
    \nabla M(y) = \bbE_{x \sim P_{X | Y = y}}\left[\nabla_{y}\log p_{X | Y= y}(x) \cdot \psi(x, y)\right] + \bbE_{x \sim P_{X | Y = y}}[\nabla_{2}\psi(x, y)]~.
\end{equation*}
\end{proposition}

The next two propositions state a ``decomposition lemma'' for the \(\Phi\)-entropy \(\bbJ^{\Phi}\), and form the starting points for the proofs of the theorems in \cref{sec:two-scale-criteria}.
Such decomposition lemmas have appeared implicitly or explicitly in prior work, focusing on specific \(\Phi\).

\begin{proposition}
\label{prop:decomp-phient}
Let \(\nu\) be the joint distribution formed by \(P_{X | Y = y}\) and \(\rho\) with density defined in \cref{eq:joint-density}.
For any \(\Phi : \calS \to \bbR\) and \(\psi : \bbR^{d} \times \bbR^{d} \to \calS\), we have
\begin{equation*}
    \Phient{\nu}{\Phi}{\psi} = \bbE_{y \sim \rho}\left[\Phient{x \sim P_{X | Y = y}}{\Phi}{\psi(x, y)}\right] + \Phient{y \sim \rho}{\Phi}{\bbE_{x \sim P_{X | Y = y}}[\psi(x, y)]}~.
\end{equation*}
When \(\psi\) is simply a function of \(x\), we obtain
\begin{equation*}
    \Phient{\mu}{\Phi}{\psi} = \bbE_{y \sim \rho}\left[\Phient{P_{X | Y = y}}{\Phi}{\psi(x)}\right] + \Phient{y \sim \rho}{\Phi}{\bbE_{P_{X | Y = y}}[\psi]}~
\end{equation*}
where \(\mu\) is the mixture distribution with density defined in \cref{eq:mixture-density}.
\end{proposition}

The next lemma is central to the unifying approach and is based on the variational formulation of the \(\Phi\)-entropy \(\bbJ^{\Phi}\) (\citet[Prop. 4]{chafai2004entropies}, \citet[Lem. 1]{boucheron2005moment}).
We begin with the definition of the convex conjugate of \(\Phi\), denoted by \(\Phi^{\star}\).
\begin{equation*}
    \Phi^{\star}(y) = \sup_{x \in \calS} xy - \Phi(x)~,
\end{equation*}
By optimality, we have for any \(t \in \calS\)
\begin{equation}
\label{eq:conv-conj-identity}
    \Phi(t) + \Phi^{\star}(\Phi'(t)) = t \cdot \Phi'(t)~.
\end{equation}
Moreover, when \(\Phi\) is a Legendre type function \citep[Chap. 26]{rockafellar1997convex}, the domain of \(\Phi^{\star}\) is the range of \(\Phi'\) and it holds that \({\Phi^{\star}}'(\Phi'(t)) = t\) for all \(t \in \calS\).

\begin{proposition}
\label{prop:product-phi-entropy}
Let \(\Phi : \calS \to \bbR\) be a twice-differentiable function such that \(\frac{1}{\Phi''}\) is concave.
Then, for any probability measure \(\pi\) and suitably integrable functions \(f_{1}, f_{2}\) over \(\calS\), it holds that
\begin{equation*}
    \bbE_{\pi}[(\Phi'(f_{1}) - \bbE_{\pi}[\Phi'(f_{1})]) \cdot f_{2}] \leq \Phient{\pi}{\Phi}{f_{2}} + \bbE_{\pi}[f_{2}] \cdot (\Phi'(\bbE_{\pi}[f_{1}]) - \bbE_{\pi}[\Phi'(f_{1})]) + \Phient{\pi}{\Phi^{\star}(\Phi')}{f_{1}}~.
\end{equation*}
\end{proposition}
The above discussion and proposition is relevant since for both \(t \xrightarrow{\Phi} t^{2}\) and \(t \xrightarrow{\Phi} t\log(t)\), (1) \(\frac{1}{\Phi''}\) is concave.
These functions are also of Legendre type, the domain of their respective \(\Phi^{\star}\) is \(\bbR\).

\subsubsection{Proof of \cref{thm:Phi-general-joint}}
\label{sec:prf:joint}

\begin{proof}
To streamline the proofs, we make the assumptions that
\begin{equation*}
    P_{X | Y = y} \text{ satisfies } \PhiSI{\beta} ~\forall ~ y~;\qquad \rho \text{ satisfies } \PhiSI{\alpha}~,
\end{equation*}
As discussed previously, this corresponds to the Poincar\'{e} and log-Sobolev inequalities for certain choices of \(\Phi\), and we instantiate these later to prove the results for these \(\Phi\) specifically.

We would like to show that
\begin{equation*}
    \Phient{\nu}{\Phi}{\psi} \leq \frac{\zeta}{2} \cdot \bbE_{\nu}[\Phi''(\psi) \cdot \|\nabla \psi\|^{2}]~,
\end{equation*}
where \(\zeta\) is defined in \cref{thm:Phi-general-joint}.
We begin with \cref{prop:decomp-phient}, which gives us
\begin{equation*}
    \Phient{\nu}{\Phi}{\psi} = \underbrace{\bbE_{y \sim \rho}\left[\Phient{x \sim P_{X | Y = y}}{\Phi}{\psi(x, y)}\right]}_{T_{1}} + \underbrace{\Phient{y \sim \rho}{\Phi}{\bbE_{x \sim P_{X | Y = y}}[\psi(x, y)]}}_{T_{2}}~.
\end{equation*}
We now study \(T_{1}\) and \(T_{2}\) individually.

\paragraph{Bounding \(T_{1}\)}
Since \(P_{X | Y = y}\) satisfies \(\PhiSI{\beta}\) for all \(y\), we have for any \(y\) that
\begin{equation*}
    \Phient{x \sim P_{X | Y =y}}{\Phi}{\psi(x, y)} \leq \frac{\beta}{2} \cdot \bbE_{x \sim P_{X | Y = y}}[\Phi''(\psi(x, y)) \cdot \|\nabla_{1}\psi(x, y)\|^{2}]~.
\end{equation*}
Consequently,
\begin{align*}
    T_{1} &= \bbE_{y \sim \rho}\left[\Phient{x \sim P_{X | Y = y}}{\Phi}{\psi(x, y)}\right] \\
    &\leq \frac{\beta}{2} \cdot \bbE_{y \sim \rho}\left[\bbE_{x \sim P_{X | Y = y}}[\Phi''(\psi(x, y)) \cdot \|\nabla_{1}\psi(x, y)\|^{2}]\right] \\
    &= \frac{\beta}{2} \cdot \bbE_{\nu}[\Phi''(\psi) \cdot \|\nabla_{1}\psi\|^{2}]~.
\end{align*}

\paragraph{Bounding \(T_{2}\)}
For convenience, we consider the following notation
\begin{equation}
\label{eq:shorthand-not}
    M(y) \equiv \bbE_{x \sim P_{X | Y = y}}[\psi(x, y)]~; \quad \calQ_{u, y} \equiv \langle u, \nabla M(y)\rangle~; \quad h_{u, y}(x) \equiv \langle u, \nabla_{y}\log p_{X | Y = y}(x)\rangle~.
\end{equation}
Since \(\rho\) satisfies \(\PhiSI{\alpha}\), we have
\begin{equation*}
    \Phient{y \sim \rho}{\Phi}{M(y)} \leq \frac{\alpha}{2} \cdot \bbE_{y \sim \rho}[\Phi''(M(y)) \cdot \|\nabla M(y)\|^{2}]~.
\end{equation*}
By the variational definition of the norm, we have
\begin{equation*}
    \|\nabla M(y)\|^{2} = \sup_{u : \|u\| \leq 1} {\calQ_{u, y}}^{2}~.
\end{equation*}
Using the expression for \(\nabla M(y)\) from \cref{prop:expected-value-kernel-test-grad}, we have for \(u\) such that \(\|u\| \leq 1\) that
\begin{align*}
    {\calQ_{u, y}}^{2} &= \left(\left\langle u, \bbE_{x \sim P_{X | Y = y}}\left[\nabla_{2}\psi(x, y)\right]\right\rangle + \left\langle u, \bbE_{x \sim P_{X | Y = y}}\left[\nabla_{y}\log p_{X | Y = y}(x) \cdot \psi(x, y)\right]\right\rangle\right)^{2} \\
    &\overset{(a)}\leq (1 + \sfC) \cdot \left\langle u, \bbE_{x \sim P_{X | Y = y}}\left[\nabla_{2}\psi(x, y)\right]\right\rangle^{2} \\
    &\qquad + (1 + \sfC^{-1}) \cdot \left\langle u, \bbE_{x \sim P_{X | Y = y}}\left[\nabla_{y}\log p_{X | Y = y}(x) \cdot \psi(x, y)\right]\right\rangle^{2} \\
    &\overset{(b)}\leq (1 + \sfC) \cdot \bbE_{x \sim P_{X | Y = y}}[\|\nabla_{2}\psi(x, y)\|]^{2} \\
    &\qquad + (1 + \sfC^{-1}) \cdot \left\langle u, \bbE_{x \sim P_{X | Y = y}}\left[\nabla_{y}\log p_{X | Y = y}(x) \cdot \psi(x, y) \right]\right\rangle^{2}~.
\end{align*}
Step \((a)\) is due to Young's inequality (\cref{prop:youngs-square}).
Step \((b)\) uses that fact that \(\|u\| \leq 1\), and the convexity of the Euclidean norm.
Therefore, we have for \(u\) such that \(\|\nabla M(y)\|^{2} = {\calQ_{u, y}}^{2}\) that
\begin{align*}
\Phient{\nu}{\Phi}{\psi} &\leq \underbrace{\frac{\beta}{2} \cdot \bbE_{\nu}[\Phi''(\psi) \cdot \|\nabla_{1}\psi\|^{2}]}_{T_{1}'} + \underbrace{\frac{\alpha \cdot (1 + \sfC)}{2} \cdot \left\{\bbE_{y \sim \rho}\left[\Phi''(M(y)) \cdot \bbE_{x \sim P_{X | Y = y}}\left[\|\nabla_{2}\psi(x, y)\|\right]^{2}\right]\right\}}_{T_{2}'} \\
&~ + \underbrace{\frac{\alpha \cdot (1 + \sfC^{-1})}{2} \cdot \left\{\bbE_{y \sim \rho}\left[\Phi''(M(y)) \cdot \left\langle u, \bbE_{x \sim P_{X | Y = y}}\left[\nabla_{y}\log p_{X | Y = y}(x) \cdot \psi(x, y) \right]\right\rangle^{2}\right]\right\}}_{T_{3}'}~.\numberthis\label{eq:prelim-T1-T2-sum-bound-joint}
\end{align*}

The term \(T_{1}'\) requires no further simplification as it is of the desired form.
We simplify \(T_{2}'\) and \(T_{3}'\).
\begin{align*}
    \bbE_{x \sim P_{X | Y = y}}\left[\|\nabla_{2}\psi(x, y)\|\right]^{2} &= \bbE_{x \sim P_{X | Y = y}}\left[\|\nabla_{2}\psi(x, y)\| \cdot \frac{\sqrt{\Phi''(\psi(x, y))}}{\sqrt{\Phi''(\psi(x, y))}}\right]^{2} \\
    &\overset{(a)}\leq \bbE_{x \sim P_{X | Y = y}}\left[\Phi''(\psi(x, y)) \cdot \|\nabla_{2}\psi(x, y)\|^{2}\right] \cdot \bbE_{x \sim P_{X | Y = y}}\left[\frac{1}{\Phi''(\psi(x, y))}\right] \\
    &\overset{(b)}\leq \bbE_{x \sim P_{X | Y = y}}\left[\Phi''(\psi(x, y)) \cdot \|\nabla_{2}\psi(x, y)\|^{2}\right] \cdot \frac{1}{\Phi''\left(\bbE_{x \sim P_{X | Y = y}}[\psi(x, y)]\right)}~\\
    &= \bbE_{x \sim P_{X | Y = y}}\left[\Phi''(\psi(x, y)) \cdot \|\nabla_{2}\psi(x, y)\|^{2}\right] \cdot \frac{1}{\Phi''(M(y))}~.
\end{align*}
In step \((a)\), we use the Cauchy-Schwarz inequality, and in step \((b)\) we use the concavity of \(\frac{1}{\Phi''}\).
This provides a simplified bound for \(T_{2}'\) as follows.
\begin{equation}
\label{eq:T2dash-bound}
    T_{2}' \leq \frac{\alpha \cdot (1 + \sfC)}{2} \cdot \bbE_{\nu}\left[\Phi''(\psi) \cdot \|\nabla_{2}\psi\|^{2}\right]~.
\end{equation}

We handle \(T_{3}'\) in the three settings described in this statement of the theorem.
Since \(\|u\| \leq 1\),
\begin{align*}
    \left\langle u, \bbE_{x \sim P_{X | Y = y}}\left[\nabla_{y}\log p_{X | Y = y}(x) \cdot \psi(x, y)\right]\right\rangle^{2} &\leq \left\|\bbE_{x \sim P_{X | Y = y}}\left[\nabla_{y}\log p_{X | Y = y}(x) \cdot \psi(x, y)\right]\right\|^{2} \\
    &\overset{(a)}\leq \bar{L}^{2} \cdot \bbE_{x \sim P_{X | Y = y}}\left[\|\nabla_{1}\psi(x, y)\|^{2}\right]~.
\end{align*}
In step \((a)\), we apply \ref{eq:2scale-exc} for \(\bfP\).
In this same manner as dealt with \(T_{2}'\), when \(\frac{1}{\Phi''}\) is concave
\begin{equation*}
    \bbE_{x \sim P_{X | Y = y}}\left[\|\nabla_{1}\psi(x, y)\|^{2}\right] \leq \bbE_{x \sim P_{X | Y = y}}\left[\Phi''(\psi(x, y)) \cdot \|\nabla_{1}\psi(x, y)\|^{2}\right] \cdot \frac{1}{\Phi''(M(y))}~.
\end{equation*}
As a result, we obtain
\begin{align*}
    T_{3}' &\leq \frac{\alpha \cdot (1 + \sfC^{-1}) \cdot \bar{L}^{2}}{2} \cdot \left\{\bbE_{y \sim \rho}\left[\bbE_{x \sim P_{X | Y = y}}[\Phi''(\psi(x, y)) \cdot \|\nabla_{1}\psi(x, y)\|^{2}\right]\right\} \\
    &= \frac{\alpha \cdot (1 + \sfC^{-1}) \cdot \bar{L}^{2}}{2} \cdot \bbE_{x \sim \mu}\left[\Phi''(\psi(x, y)) \cdot \|\nabla_{1}\psi(x, y)\|^{2}\right]~.\numberthis\label{eq:T3dash-bound-phisi}
\end{align*}

The remainder of the proof handles the specific cases for \(\Phi\) for \(T_{3}'\).
We remark that both choices of \(\Phi\) are of Legendre type and the range of \(\Phi'\) is \(\bbR\), which we use in the following development.
Note that if \(u\) satisfies \(\|\nabla M(y)\|^{2} = {\calQ_{u, y}}^{2}\), then \(-u\) also satisfies this equality.
Hence, without loss of generality,
\begin{equation*}
    \left\langle u, \bbE_{x \sim P_{X | Y = y}}\left[\nabla_{y}\log p_{X | Y = y}(x) \cdot \psi(x, y) \right]\right\rangle \geq 0
\end{equation*}
Using the fact that \(\bbE_{x \sim P_{X | Y = y}}[\nabla_{y}\log p_{X | Y = y}(x)] = 0\), we have for \(\lambda > 0\) that
\begin{align*}
    \left\langle u, \bbE_{x \sim P_{X | Y = y}}\left[\nabla_{y}\log  \!\!\right.\right. & \!\!\!\left.\left.\phantom{{}_{{}_{|}}} p_{X | Y = y}(x) \cdot\psi(x, y) \right]\right\rangle \\
    &= \bbE_{x \sim P_{X | Y = y}}[(h_{u, y}(x) - \bbE_{P_{X | Y = y}}[h_{u, y}]) \cdot \psi(x, y)] \\
    &= \bbE_{x \sim P_{X | Y = y}}\left[\left(\lambda \cdot h_{u, y}(x)- \bbE_{P_{X | Y= y}}[\lambda \cdot h_{u, y}]\right) \cdot \left\{\frac{\psi(x, y)}{\lambda}\right\} \right] \\
    &\leq \bbE_{x \sim P_{X | Y = y}}\left[\frac{\psi(x, y)}{\lambda}\right] \cdot \left(\Phi'(\bbE_{P_{X | Y =y}}[{\Phi^{\star}}'(\lambda \cdot h_{u, y})]) - \bbE_{P_{X | Y = y}}[\lambda \cdot h_{u, y}]\right) \\
    &\qquad + \Phient{x \sim P_{X | Y = y}}{\Phi}{\frac{\psi(x, y)}{\lambda}} + \Phient{P_{X | Y = y}}{\Phi^{\star}(\Phi')}{{\Phi^{\star}}'(\lambda \cdot h_{u, y})}~. 
\end{align*}
Recall the notation from \cref{eq:shorthand-not}.
The final inequality instantiates \cref{prop:product-phi-entropy} with \(f \leftarrow {\Phi^{\star}}'(\lambda \cdot h_{u, y})\) and \(g \leftarrow \frac{\psi}{\lambda}\).
Now we branch out based on \(\Phi\).

\paragraph{When \(\Phi(t) = t^{2},~\calS = \bbR\): }
In this setting,
\begin{equation*}
    \Phi'(t) = 2t; \quad \Phi^{\star}(t) = \frac{t^{2}}{4}; \quad {\Phi^{\star}}'(t) = \frac{t}{2}~.
\end{equation*}
Consequently,
\begin{equation*}
    \Phi'(\bbE_{P_{X | Y = y}}[{\Phi^{\star}}'(\lambda \cdot h_{u, y})]) = \bbE_{P_{X | Y = y}}[\lambda \cdot h_{u, y}]~; \quad \Phient{P_{X | Y = y}}{\Phi^{\star}(\Phi')}{{\Phi^{\star}}'(\lambda \cdot h_{u, y})} = \Phient{P_{X | Y = y}}{\Phi}{\frac{\lambda \cdot h_{u, y}}{2}}~.
\end{equation*}
This results in
\begin{align*}
    \left\langle u, \bbE_{x \sim P_{X | Y = y}}\left[\nabla_{y}\log p_{X | Y = y}(x) \cdot \psi(x, y) \right]\right\rangle &\leq \var{}_{P_{X | Y = y}}\left[\frac{\lambda \cdot h_{u, y}}{2}\right] + \var{}_{x \sim P_{X | Y = y}}\left[\frac{\psi(x, y)}{\lambda}\right] \\
    &\leq \frac{\lambda^{2}}{4} \cdot \var{}_{P_{X | Y = y}}[h_{u, y}] + \frac{1}{\lambda^{2}} \cdot \var{}_{x \sim P_{X | Y = y}}[\psi(x, y)]~.
\end{align*}
Since this bound holds for any \(\lambda > 0\), optimising for \(\lambda > 0\) yields
\begin{align*}
    \left\langle u, \bbE_{x \sim P_{X | Y = y}}\left[\nabla_{y}\log p_{X | Y = y}(x) \cdot \psi(x, y) \right]\right\rangle &\leq \sqrt{\var{}_{P_{X | Y = y}}[h_{u, y}] \cdot \var{}_{x \sim P_{X | Y =y}}[\psi(x, y)]} \\
    &\leq \sqrt{\beta \cdot \bbE_{x \sim P_{X | Y = y}}[\|\nabla_{1}\psi(x, y)\|^{2}] \cdot \bar{L}^{2}}~
\end{align*}
where the final step uses the assumption that \(P_{X | Y = y}\) satisfies \ref{eq:2scale-bvar} and \(P_{X | Y = y}\) satisfies \(\PI{\beta}\).
Therefore,
\begin{align*}
    T_{3}' &\leq \alpha \cdot (1 + \sfC^{-1}) \cdot \bbE_{y \sim \rho}\left[\beta \cdot \bbE_{x \sim P_{X | Y =y}}[\|\nabla_{1}\psi(x, y)\|^{2}] \cdot \bar{L}^{2}\right] \\
    &= \alpha \cdot (1 + \sfC^{-1}) \cdot \beta \cdot \bar{L}^{2} \cdot \bbE_{\nu}\left[\|\nabla_{1}\psi\|^{2}\right] \\
    &= \frac{\alpha \cdot (1 + \sfC^{-1}) \cdot \beta \cdot \bar{L}^{2}}{2} \cdot \bbE_{\nu}\left[\Phi''(\psi) \cdot \|\nabla_{1}\psi\|^{2}\right]~.\numberthis\label{eq:T3dash-bound-pi}
\end{align*}

\paragraph{When \(\Phi(t) = t\log(t),~\calS = (0, \infty)\): }
In this case,
\begin{equation*}
    \Phi'(t) = 1 + \log(t); \quad \Phi^{\star}(t) = e^{t - 1}; \quad {\Phi^{\star}}'(t) = e^{t- 1}~.
\end{equation*}
Consequently,
\begin{align*}
    \Phi'(\bbE_{P_{X | Y = y}}[{\Phi^{\star}}'(\lambda \cdot h_{u, y})]) - \bbE_{P_{X | Y = y}}[\lambda \cdot h_{u, y}] &= \log\left(\bbE_{P_{X | Y = y}}\left[\exp\left\{\lambda \cdot h_{u, y} - \bbE_{P_{X | Y = y}}[\lambda \cdot h_{u, y}]\right\}\right]\right)~;\\
    \Phient{P_{X | Y= y}}{\Phi^{\star}(\Phi')}{{\Phi^{\star}}'(\lambda \cdot h_{u, y})} &= 0~.
\end{align*}
This results in
\begin{align*}
    \left\langle u, \bbE_{x \sim P_{X | Y = y}}\left[\nabla_{y}\log  \!\!\right.\right. & \!\!\!\left.\left.\phantom{{}_{{}_{|}}} p_{X | Y = y}(x) \cdot\psi(x, y) \right]\right\rangle \\
    &\leq \ent{}_{x \sim P_{X | Y = y}}\left[\frac{\psi(x, y)}{\lambda}\right] \\
    &\quad + \bbE_{x \sim P_{X | Y = y}}\left[\frac{\psi(x, y)}{\lambda}\right] \cdot \log\left(\bbE_{P_{X | Y = y}}\left[\exp\left\{\lambda \cdot h_{u, y} - \bbE_{P_{X | Y = y}}[\lambda \cdot h_{u, y}]\right\}\right]\right) \\
    &\leq \frac{1}{\lambda} \cdot \ent{}_{x \sim P_{X | Y = y}}[\psi(x, y)] + \frac{\lambda}{2} \cdot M(y) \cdot \bar{L}^{2}~.
\end{align*}
The final inequality uses the fact that \(P_{X | Y = y}\) satisfies \ref{eq:2scale-bmgf}.
Since this holds for any \(\lambda\), optimising for \(\lambda > 0\) yields
\begin{align*}
    \left\langle u, \bbE_{x \sim P_{X | Y = y}}\left[\nabla_{y}\log p_{X | Y = y}(x) \cdot \psi(x, y) \right]\right\rangle &\leq \sqrt{2 \cdot \ent{}_{x \sim P_{X | Y = y}}[\psi(x, y)] \cdot M(y) \cdot \bar{L}^{2}} \\
    &\leq \sqrt{\beta \cdot \bbE_{x \sim P_{X | Y = y}}\left[\frac{\|\nabla_{1}\psi(x, y)\|^{2}}{\psi(x, y)}\right] \cdot M(y) \cdot \bar{L}^{2}}~.
\end{align*}
Therefore,
\begin{align*}
    T_{3}' &\leq \frac{\alpha \cdot (1 + \sfC^{-1})}{2} \cdot \bbE_{y \sim \rho}\left[\frac{1}{M(y)} \cdot \beta \cdot \bbE_{x \sim P_{X | Y = y}}\left[\frac{\|\nabla_{1}\psi(x, y)\|^{2}}{\psi(x, y)}\right] \cdot M(y) \cdot \bar{L}^{2}\right] \\
    &= \frac{\alpha \cdot (1 + \sfC^{-1}) \cdot \beta \cdot \bar{L}^{2}}{2} \cdot \bbE_{\nu}\left[\frac{\|\nabla_{1}\psi\|^{2}}{\psi}\right] \\
    &= \frac{\alpha \cdot (1 + \sfC^{-1}) \cdot \beta \cdot \bar{L}^{2}}{2} \cdot \bbE_{\nu}\left[\Phi''(\psi) \cdot \|\nabla_{1}\psi\|^{2}\right]~.\numberthis\label{eq:T3dash-bound-lsi}
\end{align*}

Substituting \cref{eq:T2dash-bound,eq:T3dash-bound-phisi} in \cref{eq:prelim-T1-T2-sum-bound-joint} and recognising that \(\|\nabla \psi\|^{2} = \|\nabla_{1}\psi\|^{2} + \|\nabla_{2}\psi\|^{2}\),
we get
\begin{align*}
    \Phient{\nu}{\Phi}{\psi} &\leq \left(\frac{\beta}{2} + \frac{\alpha \cdot (1 + \sfC^{-1}) \cdot \beta \cdot \bar{L}^{2}}{2}\right) \cdot \bbE_{\nu}[\Phi''(\psi) \cdot \|\nabla_{1}\psi\|^{2}] + \frac{\alpha \cdot (1 + \sfC)}{2} \cdot \bbE_{\nu}[\Phi''(\psi) \cdot \|\nabla_{2}\psi\|^{2}]~\\
    &\leq \frac{\max\left\{\beta + \alpha \cdot (1 + \sfC^{-1}) \cdot \bar{L}^{2} \cdot \beta,~\alpha \cdot (1 + \sfC) \right\}}{2} \cdot \bbE_{\nu}\left[\Phi''(\psi) \cdot \|\nabla \psi\|^{2}\right]~
\end{align*}
To eliminate \(\sfC >0\), we consider the tighest bound, which can be obtained by setting
\begin{equation*}
    \beta + \alpha \cdot (1 + \sfC^{-1}) \cdot \bar{L}^{2} \cdot \beta = \alpha \cdot (1 + \sfC)~,
\end{equation*}
and in conclusion yields
\begin{equation*}
    \bbJ_{\nu}^{\Phi}[\psi] \leq \frac{\zeta}{2}\cdot \bbE_{\nu}[\Phi''(\psi) \cdot \|\nabla \psi\|^{2}]~
\end{equation*}
for
\begin{equation*}
    \zeta(\alpha, \beta, \bar{L}) = \frac{1}{2}\left(\alpha + \beta + \alpha \cdot \bar{L}^{2} + \sqrt{4\alpha^{2} \cdot \bar{L}^{2} + (\beta - \alpha + \alpha \cdot \bar{L}^{2})^{2}}\right)~.
\end{equation*}
Note that \cref{eq:T3dash-bound-lsi} and \cref{eq:T3dash-bound-pi} are instances of \cref{eq:T3dash-bound-phisi} with \(\bar{L} \leftarrow \sqrt{\beta} \cdot \bar{L}\), so the proof follows.
\end{proof}

\subsubsection{Proof of \cref{thm:Phi-general-mixture}}
\label{sec:prf:mixture}

\begin{proof}
As done in the previous proof, we make the assumptions that
\begin{equation*}
    P_{X | Y = y} \text{ satisfies } \PhiSI{\beta} ~\forall ~ y~;\qquad \rho \text{ satisfies } \PhiSI{\alpha}~.
\end{equation*}
Additionally, we assume that \(\Phi\) is of Legendre type such that \(\frac{1}{\Phi''}\) is concave, and that the range of \(\Phi'\) is \(\bbR\).
This is done to streamline the proofs.

We would like to show that
\begin{equation*}
    \Phient{\nu}{\Phi}{\psi} \leq \frac{\xi}{2} \cdot \bbE_{\nu}[\Phi''(\psi) \cdot \|\nabla \psi\|^{2}]~,
\end{equation*}
where \(\xi\) is defined in \cref{thm:Phi-general-mixture}.
We begin with \cref{prop:decomp-phient}, which gives us
\begin{equation*}
    \Phient{\mu}{\Phi}{\psi} = \underbrace{\bbE_{y \sim \rho}\left[\Phient{x \sim P_{X | Y = y}}{\Phi}{\psi(x)}\right]}_{T_{1}} + \underbrace{\Phient{y \sim \rho}{\Phi}{\bbE_{x \sim P_{X | Y = y}}[\psi(x)]}}_{T_{2}}~.
\end{equation*}
Much of this proof resembles the previous one, and we include details for the sake of completeness.

\paragraph{Bounding \(T_{1}\)}
Since \(P_{X | Y = y}\) satisfies \(\PhiSI{\beta}\) for all \(y\), we have for any \(y\) that
\begin{align*}
    \Phient{x \sim P_{X | Y =y}}{\Phi}{\psi(x)} &\leq \frac{\beta}{2} \cdot \bbE_{x \sim P_{X | Y = y}}[\Phi''(\psi(x)) \cdot \|\nabla\psi(x)\|^{2}]~ \\
    \Rightarrow T_{1} &\leq \frac{\beta}{2} \cdot \bbE_{\mu}\left[\Phi''(\psi) \cdot \|\nabla\psi\|^{2}\right]~.
\end{align*}

\paragraph{Bounding \(T_{2}\)}
For convenience, we again consider the following notation
\begin{equation*}
    M(y) \equiv \bbE_{x \sim P_{X | Y = y}}[\psi(x)]~; \quad \calQ_{u, y} \equiv \langle u, \nabla M(y)\rangle~; \quad V_{y} \equiv \nabla_{2}G(., y)~; \quad h_{u, y} \equiv \langle u, V_{y}\rangle~.
\end{equation*}
Since \(\rho\) satisfies \(\PhiSI{\alpha}\), we have
\begin{equation*}
    \Phient{y \sim \rho}{\Phi}{M(y)} \leq \frac{\alpha}{2} \cdot \bbE_{y \sim \rho}[\Phi''(M(y)) \cdot \|\nabla M(y)\|^{2}]~.
\end{equation*}

Handling this term in the same way as handling \(T_{3}'\) in the preceding proof in \Cref{sec:prf:joint}, albeit with \(\psi(x)\) in lieu of \(\psi(x, y)\), we have
\begin{align*}
    \Phient{\mu}{\Phi}{\psi} &\leq \frac{\beta}{2} \cdot \bbE_{\mu}[\Phi''(\psi) \cdot \|\nabla\psi\|^{2}] + \frac{\alpha \cdot \bar{L}^{2}}{2} \cdot \bbE_{\mu}[\Phi''(\psi) \cdot \|\nabla\psi\|^{2}] \\
    &= \frac{\beta + \alpha \cdot \bar{L}^{2}}{2} \cdot \bbE_{\mu}[\Phi''(\psi) \cdot \|\nabla\psi\|^{2}]~.
\end{align*}
\end{proof}

\subsubsection{Proof of \cref{lem:conv-product-measure}}
\label{sec:prf:conv-product-measure}

\begin{proof}

For any \(\Phi\) such that \(\frac{1}{\Phi''}\) is concave, we have from \cref{eq:prelim-T1-T2-sum-bound-joint} and \cref{eq:T2dash-bound} in the proof of \cref{thm:Phi-general-joint} that
\begin{equation*}
    \Phient{\nu}{\Phi}{\psi} \leq \frac{\beta}{2} \cdot \bbE_{\nu}[\Phi''(\psi) \cdot \|\nabla_{1}\psi\|^{2}] + \frac{\alpha}{2} \cdot \bbE_{y \sim \rho}[\Phi''(\psi) \cdot \|\nabla_{2}\psi\|^{2}]~.
\end{equation*}
For general bivariate functions \(\psi\), we have \(
\|\nabla \psi\|^{2} = \|\nabla_{1}\psi\|^{2} + \nabla_{2}\psi\|^{2}\), which gives
\begin{equation*}
    \Phient{\nu}{\Phi}{\psi} \leq \frac{\max\{\alpha, \beta\}}{2} \cdot \bbE_{\nu}[\Phi''(\psi) \cdot \|\nabla \psi\|^{2}]~.
\end{equation*}
This proves the first part of the lemma.
On the other hand, when \(\psi(x, y) \leftarrow f(x + y)\), it holds that \(\|\nabla_{1} \psi(x, y)\|^{2} = \|\nabla_{2} \psi(x, y)\|^{2} = \|\nabla f(x + y)\|^{2}\), and hence
\begin{equation*}
    \Phient{(x, y) \sim \nu}{\Phi}{f(x + y)} \leq \frac{\alpha + \beta}{2} \cdot \bbE_{(x, y) \sim \nu}[\Phi''(f(x + y)) \cdot \|\nabla f(x + y)\|^{2}]~.
\end{equation*}
To see the relevance to the second part of the lemma, we have for any measurable function \(g\) that \(\bbE_{z \sim \rho \ast P}[g(z)] = \bbE_{(x, y) \sim \nu}[g(x + y)]\).
This yields
\begin{equation*}
    \Phient{z \sim \rho \ast P}{\Phi}{f(z)} \leq \frac{\alpha + \beta}{2} \cdot \bbE_{z \sim \rho \ast P}\left[\Phi''(f(z)) \cdot \|\nabla f(z)\|^{2}\right]
\end{equation*}
which completes the proof for the second result of the lemma.
\end{proof}

\subsection{Proofs for \cref{prop:expected-value-kernel-test-grad,prop:decomp-phient,prop:product-phi-entropy}}
\label{sec:proofs-prop-proofs}

\subsubsection{Proof of \cref{prop:expected-value-kernel-test-grad}}
\begin{proof}
We start with the definition of \(M(y)\),
\begin{align*}
    \nabla M(y) &= \nabla_{y} \int p_{X | Y = y}(x) \psi(x, y) \rmd x \\
    &= \int \left(\nabla_{y} p_{X | Y = y}(x) \cdot \psi(x, y) + p_{X | Y = y}(x) \cdot \nabla_{2}\psi(x, y) \right)\rmd x \\
    &= \int \left(\nabla_{y} \log p_{X | Y = y}(x) \cdot \psi(x, y) + \nabla_{2}\psi(x, y)\right) p_{X | Y = y}(x) \rmd x \\
    &= \bbE_{x \sim P_{X | Y =y}}[\nabla_{y}\log p_{X | Y = y}(x) \cdot \psi(x, y))] + \bbE_{x \sim P_{X | Y =y}}[\nabla_{2}\psi(x, y)]~.
\end{align*}
\end{proof}

\subsubsection{Proof of \cref{prop:decomp-phient}}
\begin{proof}
By definition,
\begin{align*}
    \Phient{\nu}{\Phi}{\psi} &= \bbE_{y \sim \rho}\left[\bbE_{x \sim P_{X | Y = y}}[\Phi(\psi(x, y))]\right] - \Phi\left(\bbE_{y \sim \rho}\left[\bbE_{P_{X | Y = y}}[\psi(x, y)]\right]\right) \\
    &= \bbE_{y \sim \rho}\left[\bbE_{x \sim P_{X | Y = y}}[\Phi(\psi(x, y))]\right] - \bbE_{y \sim \rho}\left[ \Phi\left(\bbE_{x \sim P_{X | Y = y}}[\psi(x, y)]\right)\right] \\
    &\qquad + \bbE_{y \sim \rho}\left[ \Phi\left(\bbE_{x \sim P_{X | Y = y}}[\psi(x, y)]\right)\right] - \Phi\left(\bbE_{y \sim \rho}\left[\bbE_{P_{X | Y = y}}[\psi(x, y)]\right]\right) \\
    &= \bbE_{y \sim \rho}\left[\left\{\bbE_{x \sim P_{X | Y = y}}[\Phi(\psi(x, y))] - \Phi(\bbE_{x \sim P_{X | Y = y}}[\psi(x, y)])\right\}\right] + \Phient{y \sim \rho}{\Phi}{\bbE_{x \sim P_{X | Y = y}}[\psi(x, y)]} \\
    &= \bbE_{y \sim \rho}\left[\Phient{x \sim P_{X | Y = y}}{\Phi}{\psi(x, y)}\right] + \Phient{y \sim \rho}{\Phi}{\bbE_{x \sim P_{X | Y = y}}[\psi(x, y)]}~.
\end{align*}

For the second part of the proposition, note that
\begin{equation*}
    \Phient{\nu}{\Phi}{\psi} = \bbE_{(x, y) \sim \nu}[\Phi(\psi(x))] - \Phi(\bbE_{(x, y) \sim \nu}[\psi(x)]) = \bbE_{x \sim \mu}[\Phi(\psi(x))] - \Phi(\bbE_{\mu}[\psi])~,
\end{equation*}
since the \(y\)-marginal of \(\nu\) is \(\mu\).
\end{proof}

\subsubsection{Proof of \cref{prop:product-phi-entropy}}
\begin{proof}
From \citet[Prop. 4]{chafai2004entropies}, we have by the assumptions of this proposition that
\begin{equation*}
    \Phient{\pi}{\Phi}{f_{2}} \geq \Phient{\pi}{\Phi}{f_{1}} + \bbE_{\pi}\left[(\Phi'(f_{1}) - \Phi'(\bbE_{\pi}[f_{1}])) \cdot (f_{2} - f_{1})\right]~.
\end{equation*}

We write \(\Phi'(f_{1}) - \Phi'(\bbE_{\pi}[f_{1}])\) as \(\Phi'(f_{1}) - \bbE_{\pi}[\Phi'(f_{1})] + \bbE_{\pi}[\Phi'(f_{1})] - \Phi'(\bbE_{\pi}[f_{1}])\), and expand the second term in the RHS above.
\begin{align*}
    \bbE_{\pi}\left[(\Phi'(f_{1}) - \Phi'(\bbE_{\pi}[f_{1}])) \cdot (f_{2} - f_{1})\right] &= \bbE_{\pi}[(\Phi'(f_{1}) - \bbE_{\pi}[\Phi'(f_{1})]) \cdot (f_{2} - f_{1})] \\
    &\qquad - \bbE_{\pi}[f_{2} - f_{1}] \cdot (\Phi'(\bbE_{\pi}[f_{1}]) - \bbE_{\pi}[\Phi'(f_{1})]) \\
    &= \bbE_{\pi}[(\Phi'(f_{1}) - \bbE_{\pi}[\Phi'(f_{1})]) \cdot f_{2}] \\
    &\quad - \bbE_{\pi}[f_{2}] \cdot (\Phi'(\bbE_{\pi}[f_{1}]) - \bbE_{\pi}[\Phi'(f_{1})]) \\
    &\qquad - \bbE_{\pi}[\Phi'(f_{1}) \cdot f_{1}] + \bbE_{\pi}[f_{1}] \cdot \Phi'(\bbE_{\pi}[f_{1}])~.\numberthis\label{eq:var-form-second-term}
\end{align*}

Instantiating the identity in \cref{eq:conv-conj-identity} for the last two terms in \cref{eq:var-form-second-term}, we get
\begin{align*}
    \bbE_{\pi}[f_{1}] \cdot \Phi'(\bbE_{\pi}[f_{1}]) &= \Phi(\bbE_{\pi}[f_{1}]) + \Phi^{\star}(\Phi'(\bbE_{\pi}[f_{1}])) \\
    \bbE_{\pi}[\Phi'(f_{1}) \cdot f_{1}] &= \bbE_{\pi}[\Phi(f_{1})] + \bbE_{\pi}[\Phi^{\star}(\Phi'(f_{1}))]
\end{align*}
This results in
\begin{align*}
    \bbE_{\pi}[f_{1}] \cdot \Phi'(\bbE_{\pi}[f_{1}]) - \bbE_{\pi}[\Phi'(f_{1}) \cdot f_{1}] &= \Phi(\bbE_{\pi}[f_{1}]) - \bbE_{\pi}[\Phi(f_{1})] + \Phi^{\star}(\Phi'(\bbE_{\pi}[f_{1}])) - \bbE_{\pi}[\Phi^{\star}(\Phi'(f_{1}))]  \\
    &= -\Phient{\pi}{\Phi}{f_{1}} + \Phi^{\star}(\Phi'(\bbE_{\pi}[f_{1}])) - \bbE_{\pi}[\Phi^{\star}(\Phi'(f_{1}))] \\
    &= -\Phient{\pi}{\Phi}{f_{1}} - \Phient{\pi}{\Phi^{\star}(\Phi')}{f_{1}}~.\numberthis\label{eq:conv-conj-simplify}
\end{align*}

Substituting \cref{eq:var-form-second-term,eq:conv-conj-simplify} in the inequality due to \citeauthor{chafai2004entropies}, we obtain that
\begin{equation*}
    \Phient{\pi}{\Phi}{f_{2}} \geq \bbE_{\pi}[(\Phi'(f_{1}) - \bbE_{\pi}[\Phi'(f_{1})]) \cdot f_{2}] -  \bbE_{\pi}[f_{2}] \cdot (\Phi'(\bbE_{\pi}[f_{1}]) - \bbE_{\pi}[\Phi'(f_{1})]) - \Phient{\pi}{\Phi^{\star}(\Phi')}{f_{1}}~,
\end{equation*}
which when rearranged gives the statement of the proposition.
\end{proof}

\subsection{Proofs for the results in \Cref{sec:applications}}

\subsubsection{Proof of \Cref{thm:ricci-curvature}}
\label{sec:prf:ricci-curvature}
\begin{proof}

The proof for \(\kappa_{1}(y_{1}, y_{2})\) is more straightforward than the result for \(\kappa_{2}(y_{1}, y_{2})\).
We state the result for \(\kappa_{1}(y_{1}, y_{2})\) as a warm-up first.

\textbf{Warm-up: Result for \(\kappa_{1}(y_{1}, y_{2})\).}

Let \(g\) be an arbitrary \(1\)-Lipschitz continuous function.
Define \(y_{t} = y_{1} + t \cdot (y_{2} - y_{1})\) for \(t \in [0, 1]\).
Then
\begin{equation}
\label{eq:mvt-cond}
    p_{X | Y = y_{1}}(x) - p_{X | Y = y_{2}}(x) = \int_{0}^{1} \langle \nabla_{y} p_{X | Y = y_{\tilde{t}}}(x), y_{2} - y_{1}\rangle \rmd t~.
\end{equation}
This results in
\begin{align*}
    \left|\int g(x) p_{X | Y = y_{1}}(x) \rmd x \right.&\left.- \int g(x) p_{X | Y = y_{2}}(x) \rmd x\right| \\
    &= \left|\int \int_{0}^{1} g(x) \langle \nabla_{y} p_{X | Y = y_{t}}(x), y_{2} - y_{1}\rangle \rmd t ~\rmd x\right|\\
    &\overset{(a)}= \left|\int_{0}^{1} \left\langle y_{1} - y_{2}, \int \nabla_{y} p_{X | Y = y_{t}}(x) g(x)\rmd x \right\rangle \rmd t\right| \\
    &\overset{(b)}\leq \int_{0}^{1}\left| \left\langle y_{1} - y_{2}, \int \nabla_{y} p_{X | Y = y_{t}}(x) g(x)\rmd x \right\rangle \right| \rmd t \\
    &\overset{(c)}= \int_{0}^{1} \left| \langle y_{1} - y_{2}, \bbE_{\sfx \sim P_{X | Y = y_{t}}}\left[g(\sfx) \cdot \nabla_{y} \log p_{X | Y = y_{t}}(\sfx)\right]\rangle\right| \rmd t \\
    &\overset{(d)}\leq \int_{0}^{1} \bar{L} \cdot \sqrt{\bbE_{\sfx \sim P_{X | Y = y_{t}}}[\|\nabla g(\sfx)\|^{2}]} \cdot \|y_{1} - y_{2}\| \rmd t \\
    &\overset{(e)}\leq \bar{L} \cdot \|y_{1} - y_{2}\|~.
\end{align*}
The sequence of steps \((a)-(e)\) are justified in order as follows: Tonelli's theorem for interchanging the order of integration, Jensen's inequality for moving the absolute value inside the integral, using chain rule to express the quantity in terms of the LHS in \ref{eq:2scale-exc}, application of \ref{eq:2scale-exc}, and finally noting that \(g\) is \(1\)-Lipschitz continuous.

Since this holds for an arbitrary \(1\)-Lipschitz continuous function, we have from \cref{prop:W1-dual} that 
\begin{equation*}
    \sfW_{1}(p_{X | Y = y_{1}}, p_{X | Y = y_{2}}) \leq \bar{L} \cdot \|y_{1} - y_{2}\|~.
\end{equation*}
The statement of the lemma is thus obtained by definition.

\textbf{Result for \(\kappa_{2}(y_{1}, y_{2})\).}

Let \(g\) be an arbitrary bounded continuous function.
Define
\begin{equation*}
    u(x, t) = \argmin_{z} g(z) + \frac{1}{2t}\|x - z\|^{2}~.
\end{equation*}

We have by the duality of \(\sfW_{2}\) that
\begin{equation*}
    \frac{1}{2}\sfW_{2}^{2}(p_{X | Y = y_{1}}, p_{X | Y = y_{2}}) \leq \int u(x, 1) p_{X | Y = y_{2}}(x) \rmd x - \int u(x, 0) p_{X | Y = y_{1}}(x) \rmd x = h(1) - h(0)~.
\end{equation*}
where the functional \(h\) is defined as \(h(t) := \int u(x, t) p_{X | Y = y_{t}}(x)\rmd x\)~.
From the Hopf-Lax formula \citep[Theorem 4, Chap. 3]{evans2010partial}, we have that
\begin{equation*}
    \frac{\rmd}{\rmd t} u(x, t) + \frac{1}{2}\|\nabla u(x, t)\|^{2} = 0~
\end{equation*}
Consequently,
\begin{align*}
    \frac{\rmd}{\rmd t} h(t) &= \int \left\{\frac{\rmd}{\rmd t} u(x, t) p_{X | Y = y_{t}}(x) + u(x, t) \left\langle \nabla p_{X | Y = y_{t}}, y_{2} - y_{1}\right\rangle\right\}\rmd x \\
    &\overset{(a)}= -\frac{1}{2} \bbE_{x \sim p_{X | Y = y_{t}}}\left[\|\nabla u(x, t)\|^{2}\right] + \bbE_{x \sim p_{X | Y = y_{t}}}\left[u(x, t) \langle \nabla \log p_{X | Y = y_{t}}(x), y_{2} - y_{1}\rangle\right] \\
    &\overset{(b)}\leq -\frac{1}{2}\bbE_{x \sim p_{X | Y = y_{t}}}[\|\nabla u(x, t)\|^{2}] + \bar{L} \cdot \|y_{1} - y_{2}\| \cdot \sqrt{\bbE_{x \sim P_{X | Y = y_{t}}}[\|\nabla u(x, t)\|^{2}]} \\
    &\overset{(c)}\leq \frac{\bar{L}^{2} \cdot \|y_{1} - y_{2}\|^{2}}{2}~.
\end{align*}
Step \((a)\) substitutes the Hopf-Lax formula, step \((b)\) uses \(\nabla f(x) = f(x) \cdot \nabla \log f(x)\), and step \((c)\) uses the definition of \ref{eq:2scale-exc} and the AM-GM inequality.
As a result,
\begin{equation*}
    \frac{1}{2}\sfW_{2}^{2}(p_{X | Y = y_{1}}, p_{X | Y = y_{2}}) \leq h(1) - h(0) = \int_{0}^{1} \frac{\rmd}{\rmd t}h(t) \rmd t \leq \frac{\bar{L}^{2} \cdot \|y_{1} - y_{2}\|^{2}}{2}~.
\end{equation*}
Multiplying both sides by \(2\) and taking the square root yields the statement about \(\kappa_{2}(y_{1}, y_{2})\), via \cref{prop:W2-dual} and since \(g\) is arbitrary.
\end{proof}

\subsubsection{Proof of \Cref{lem:conv-error-bound}}
\label{sec:prf:conv-error-bound}

\begin{proof}
We show the results for schemes 1 and 2 sequentially.
Suppose \(\sfX\) is obtained after \(M\) transitions with \(\bfP\) from initial distribution \(\rho_{0}\).
Then \(\sfX \sim \rho_{M}\), where \(\rho_{M}\) is a mixture of \(\bfP\) with mixing distribution \(\rho_{M - 1}\).
Since \(\bar{L} < 1\), we know from \cref{prop:w1-contract} that \(\rho^{\star}\) is stationary, and hence a mixture of \(\bfP\) with mixing distribution \(\rho^{\star}\) coincides with \(\rho^{\star}\).
As a result, from \cref{prop:W1-dual},
\begin{align*}
    \left|\bbE[f(\sfX)] - F^{\star}(f; \rho^{\star})\right| &= \left|F^{\star}(f; \rho_{M}) - F^{\star}(f; \rho^{\star})\right| \\
    &\leq L \cdot \mathsf{W}_{1}(\rho_{M}, \rho^{\star}) \\
    &\leq L \cdot \bar{L} \cdot \mathsf{W}_{1}(\rho_{M - 1}, \rho^{\star}) \\
    &\leq L \cdot \bar{L}^{M} \cdot \mathsf{W}_{1}(\rho_{0}, \rho^{\star})
\end{align*}
where in the penultimate inequality, we use \cref{prop:w1-contract} and the last inequality, we recurse the inequality.
We use this intermediate result to bound the errors for Scheme 1 and 2 below.

\paragraph{Scheme 1}
Note that each \(x \in \calS_{N}\) is iid. according to \(\rho_{J}\).
By linearity of expectation,
\begin{equation*}
    \left|\bbE[\widehat{F}(f; \calS_{N})] - F^{\star}(f; \rho^{\star})\right| = \left|\frac{1}{N}\sum_{i=1}^{N}\left\{F^{\star}(f; \rho_{J}) - F^{\star}(f; \rho^{\star})\right\}\right| \leq L \cdot \bar{L}^{J} \cdot \mathsf{W}_{1}(\rho_{0}, \rho^{\star})~.
\end{equation*}

\paragraph{Scheme 2}
The first sample in \(\calS_{N}\) is distributed according to \(\rho_{J}\), and the \(i^{th}\) sample after that is distribution according to \(\rho_{J + (i - 1)K}\).
By the linearity of expectation,
\begin{align*}
    \left|\bbE[\widehat{F}(f; \calS_{N})] - F^{\star}(f; \rho^{\star})\right| &= \left|\frac{1}{N}\sum_{i=1}^{N} \left\{F^{\star}(f; \rho_{J + (i - 1)K}) - F^{\star}(f; \rho^{\star})\right\}\right| \\
    &\leq \frac{1}{N}\sum_{i=1}^{N}\left|F^{\star}(f; \rho_{J + (i - 1)K}) - F^{\star}(f; \rho^{\star})\right| \\
    &\leq \frac{L \cdot \bar{L}^{J} \cdot \mathsf{W}_{1}(\rho_{0}, \rho^{\star})}{N}\sum_{i=1}^{N} \bar{L}^{(i - 1)K}~.
\end{align*}
\end{proof}

\subsubsection{Proof of \Cref{lem:scheme1-chernoff}}
\label{sec:prf:scheme1-chernoff}

\begin{proof}
Since the samples in \(\calS_{N}\) are independent, we have
\begin{align*}
    \bbE\left[\exp\left(\lambda \cdot \left\{\widehat{F}(f, \calS_{N}) - \bbE\left[\widehat{F}(f, \calS_{N})\right]\right\}\right)\right] &= \prod_{i=1}^{N} \bbE\left[\exp\left(\frac{\lambda}{N} \cdot \left\{ f(x^{(i)}) - \bbE\left[f(x^{(i)})\right]\right\}\right)\right] \\
    &= \prod_{i=1}^{N} \bbE\left[\exp\left(\frac{\lambda \cdot L}{N} \cdot \frac{1}{L}\left\{ f(x^{(i)}) - \bbE\left[f(x^{(i)})\right]\right\}\right)\right]
\end{align*}
Since each \(x^{(i)}\) for \(i \in [N]\) is obtained by running the Markov chain for \(J\) iterations from \(\rho_{0}\) that satisfies \(\PhiSI{\alpha_{0}}\), the law of \(x^{(i)}\) for every \(i \in [N]\) satisfies \(\PhiSI{\alpha_{J}}\)
where
\begin{equation*}
    \alpha_{J} = \beta \sum_{i=0}^{J - 1} \bar{L}^{2i} + \alpha_{0} \bar{L}^{2J}
\end{equation*}
which is a consequence of \cref{thm:Phi-general-mixture}.
As a result, when \(\frac{\lambda \cdot L}{N} < \Lambda\),
\begin{equation*}
    \prod_{i=1}^{N} \bbE\left[\exp\left(\frac{\lambda \cdot L}{N} \cdot \frac{1}{L}\left\{ f(x^{(i)}) - \bbE\left[f(x^{(i)})\right]\right\}\right)\right] \leq \varphi_{\Phi}\left(\frac{\lambda \cdot L}{N}; \alpha_{J}\right)^{N}
\end{equation*}
Therefore,
\begin{equation*}
    \bbP\left(\widehat{F}(f; \calS_{N}) - \bbE[\widehat{F}(f; \calS_{N})] > t\right) \leq \inf_{\lambda : \lambda < \frac{\Lambda N}{L}} \varphi_{\Phi}\left(\frac{\lambda \cdot L}{N}; \alpha_{J}\right)^{N} \exp\left(-\lambda t\right)~.
\end{equation*}

We instantiate this with the choices of \(\varphi_{\Phi}\) and \(\Lambda\) highlighted in \cref{assump:mgf}.
\begin{description}
\item [\(\Phi(t) = t\log(t)\): ] the infimum on the right side above turns into
\begin{equation*}
    \inf_{\lambda > 0} \exp\left(\frac{\alpha_{J}\lambda^{2}L^{2}}{2N} - \lambda t\right) = \exp\left(-\frac{Nt^{2}}{2\alpha_{J}L^{2}}\right)
\end{equation*}
since the minimum of \(\frac{\alpha_{J}\lambda^{2}L^{2}}{2N} - \lambda t\) over \(\lambda > 0\) is \(-\frac{Nt^{2}}{2\alpha_{J}L^{2}}\).

\item [\(\Phi(t) = t^{\nicefrac{2}{p}}\): ] this is a more involved algebraic calculation.
First, by substituting the definition of \(\varphi_{\Phi}\) in this setting, we have
\begin{equation*}
    \varphi_{\Phi}\left(\frac{\lambda \cdot L}{N}; \alpha_{J}\right)^{N} \exp(-\lambda t) = \exp\left(-\lambda t - \frac{2N}{2 - p}\log\left(1 - \frac{\alpha_{J}\lambda^{2}L^{2}(2 - p)}{4N^{2}}\right)\right)~.
\end{equation*}
Taking the derivative w.r.t. \(\lambda\) and setting it to \(0\) (without considering constraints), we have
\begin{equation*}
    -t + \left(1 - \frac{\alpha_{J}\lambda^{2}L^{2}(2 - p)}{4N^{2}}\right)^{-1}\cdot \frac{\alpha_{J}\lambda L}{N} =0~.
\end{equation*}
Solving this quadratic yields
\begin{equation*}
    \lambda^{\star} = \frac{2N}{(2 - p)L}\left(\sqrt{\frac{1}{t^{2}} + \frac{2 - p}{\alpha_{J}}} - \frac{1}{t}\right)~.
\end{equation*}
By the subadditivity of \(t \mapsto \sqrt{t}\), we also know that \(\lambda^{\star} < \frac{2N}{L\sqrt{\alpha_{J}(2 - p)}}\), respecting the constraint.
Substituting this in the function we sought to minimise, we have for \(\bar{\alpha}_{J}(t; p) := \sqrt{1 + \frac{t^{2}(2 - p)}{\alpha_{J}L^{2}}}\)
\begin{equation*}
    \inf_{\lambda < \Lambda}  \varphi_{\Phi}\left(\frac{\lambda}{N}; \alpha_{J}\right)^{N} \exp(-\lambda t) = \exp\left(-\frac{2N}{2 - p}\left\{\bar{\alpha}_{J}(t; p) - 1 - \log(\bar{\alpha}_{J}(t; p) + 1) + \log(2)\right\}\right)~.
\end{equation*}
\end{description}
\end{proof}

\subsubsection{Proof of \Cref{lem:scheme2-chernoff}}
\label{sec:prf:scheme2-chernoff}

\begin{proof}
Here, we have a dependency between samples.
However, due to the Markovian nature of the collection process, the sample \(x^{(i)}\) only depends on \(x^{(i - 1)}\) for any \(i \geq 2\).

For convenience, let \(\calF_{n}\) represents the random variables \(\{x^{(i)}\}_{i=1}^{n}\), and define 
\begin{align*}
    \calP_{n} &:= \bbE_{\calF_{n}}\left[\exp\left(\frac{\lambda \cdot L}{N} \cdot \frac{1}{L} \cdot \sum_{i=1}^{n}\left\{ f(x^{(i)}) - \bbE_{\calF_{N}}[f(x^{(i)})]\right\}\right)\right]~, \\
    \calA_{n} &:= \bbE_{\calF_{n - 1}}\left[\exp\left(\frac{\lambda}{N} \cdot \left\{\bbE_{x^{(n)} | \calF_{n - 1}}[f(x^{(n)})] - \bbE_{\calF_{n}}[f(x^{(n)})]\right\}\right)\right] \\
    &= \bbE_{\calF_{n - 1}}\left[\exp\left(\frac{\lambda}{N} \cdot \left\{\bbE_{x^{(n)} | \calF_{n - 1}}[f(x^{(n)})] - \bbE_{\calF_{n - 1}}\left[\bbE_{x^{(n)} | \calF_{n - 1}}[f(x^{(n)})]\right]\right\}\right)\right] \\
    &= \bbE_{x^{(n - 1)}}\left[\exp\left(\frac{\lambda}{N} \cdot \left\{\bbE_{x^{(n)} | x^{(n - 1)}}[f(x^{(n)})] - \bbE_{x^{(n - 1)}}\left[\bbE_{x^{(n)} | x^{(n - 1)}}[f(x^{(n)})]\right]\right\}\right)\right]~.
\end{align*}
The last equality in \(\calA_{n}\) stems from the fact that \(x^{(n)}\) is only dependent on \(x^{(n - 1)}\), and \(\bbE_{x^{(n)}}\) represents the marginal law of \(x^{(n - 1)}\).
This leads to the following mathematical observation
\begin{align*}
   \calP_{N} &= \bbE_{\calF_{N}}\left[\exp\left(\frac{\lambda \cdot L}{N} \cdot \frac{1}{L}\sum_{i=1}^{N}\left\{ f(x^{(i)}) - \bbE_{\calF_{N}}[f(x^{(i)})]\right\}\right)\right] \\
   &= \underbrace{\bbE_{\calF_{N - 1}}\left[\bbE_{x^{(N)} | \calF_{N - 1}}\left[\exp\left(\frac{\lambda \cdot L}{N} \cdot \frac{1}{L}\left\{f(x^{(N)}) - \bbE_{\calF_{N}}[f(x^{(N)}]\right\}\right)\right]\right]}_{T_{N}} \cdot \calP_{N - 1}~.
\end{align*}

We now work with the first term.
Adding and subtracting \(\bbE_{x^{(N)} | \calF_{N - 1}}[f(x^{(N)})]\), we get
\begin{align*}
    T_{N} &= \bbE_{\calF_{N - 1}}\left[\bbE_{x^{(N)} | \calF_{N - 1}}\left[\exp\left(\frac{\lambda \cdot L}{N} \cdot \frac{1}{L}\left\{f(x^{(N)}) - \bbE_{x^{(N)} | \calF_{N - 1}}[f(x^{(N)})]\right\}\right)\right]\right] \\
    &\qquad \cdot \bbE_{\calF_{N - 1}}\left[\exp\left(\frac{\lambda}{N} \cdot \left\{\bbE_{x^{(N)} | \calF_{N - 1}}[f(x^{(N)})] - \bbE_{\calF_{N}}[f(x^{(N)})]\right\}\right)\right] \\
    &= \bbE_{x^{(N - 1)}}\left[\bbE_{x^{(N)} | x^{(N - 1)}}\left[\exp\left(\frac{\lambda \cdot L}{N} \cdot \frac{1}{L}\left\{f(x^{(N)}) - \bbE_{x^{(N)} | x^{(N - 1)}}[f(x^{(N)})]\right\}\right)\right]\right] \\
    &\qquad \cdot \bbE_{\calF_{N - 1}}\left[\exp\left(\frac{\lambda}{N} \cdot \left\{\bbE_{x^{(N)} | \calF_{N - 1}}[f(x^{(N)})] - \bbE_{\calF_{N}}[f(x^{(N)})]\right\}\right)\right] \\
    &\leq \varphi_{\Phi}\left(\frac{\lambda \cdot L}{N}; \beta\right) \cdot \calA_{N}~.
\end{align*}
The inequality uses the facts that \(x^{(N)} | x^{(N - 1)}\) is distributed according to \(p_{x | Y = x^{(N - 1)}}\), which is assumed to satisfy \(\PhiSI{\beta}\), and \cref{assump:mgf}.
Note that this assumes \(\frac{\lambda \cdot L}{N} < \Lambda\).
Recursing this inequality, we get
\begin{equation*}
    \calP_{N} \leq \varphi_{\Phi}\left(\frac{\lambda \cdot L}{N}; \beta\right)^{N - 1} \cdot \prod_{i=2}^{N} \calA_{i} \cdot \calP_{1}~.
\end{equation*}
The rest of the proof focuses on simplifying the quantities \(\{\calA_{i}\}_{i=2}^{N}\) and \(\calP_{1}\).
First, by \cref{thm:Phi-general-mixture}, the law of \(x^{(i)}\) satisfies \(\PhiSI{\alpha_{J + (i - 1) K}}\) where
\begin{align*}
    \alpha_{\ell} = \beta \sum_{i=0}^{\ell - 1} \bar{L}^{2i} + \alpha_{0}\bar{L}^{2\ell} \qquad \forall ~\ell \geq 0~.
\end{align*}
This is because \(x^{(i)}\) is obtained after \(J + (i -1)K\) simulations of the Markov chain.
Since \(f\) is \(L\)-Lipschitz continuous
\begin{equation*}
    \calP_{1} = \bbE_{x^{(1)}}\left[\exp\left(\frac{\lambda \cdot L}{N} \cdot \frac{1}{L} \left\{f(x^{(1)}) - \bbE_{x^{(1)}}\left[f(x^{(1)})\right]\right\}\right)\right] \leq \varphi_{\Phi}\left(\frac{\lambda \cdot L}{N}; \alpha_{J}\right)~.
\end{equation*}
Second, the function \(y \mapsto \bbE_{\sfx \sim P_{X | Y  =y}}[f(\sfx)]\) is \((\bar{L} \cdot L)\)-Lipschitz continuous.
This is a consequence of \Cref{thm:ricci-curvature} and \cref{prop:W1-dual}, since
\begin{align*}
    \left|\bbE_{\sfx \sim P_{X | Y = y_{1}}}[f(\sfx)] - \bbE_{\sfx \sim P_{X | Y = y_{2}}}[f(\sfx)]\right| &\leq L \cdot \sfW_{1}(P_{X | Y = y_{1}}, P_{X | Y = y_{2}}) \\
    &\leq L \cdot (1 - \kappa(y_{1}, y_{2})) \|y_{1} - y_{2}\| \\
    &\leq \bar{L} \cdot L \cdot \|y_{1} - y_{2}\|~.
\end{align*}
From this, we obtain for any \(i \geq 2\) that
\begin{align*}
    \calA_{i} &= \bbE_{x^{(i - 1)}}\left[\exp\left(\frac{\lambda}{N} \cdot \left\{\bbE_{x^{(i)} | x^{(i - 1)}}[f(x^{(i)})] - \bbE_{x^{(i - 1)}}\left[\bbE_{x^{(i)} | x^{(i - 1)}}[f(x^{(i)})]\right]\right\}\right)\right] \\
    &= \bbE_{x^{(i - 1)}}\left[\exp\left(\frac{\lambda \cdot L \cdot \bar{L}}{N} \cdot \frac{1}{L \cdot \bar{L}} \left\{\bbE_{x^{(i)} | x^{(i - 1)}}[f(x^{(i)})] - \bbE_{x^{(i - 1)}}\left[\bbE_{x^{(i)} | x^{(i - 1)}}[f(x^{(i)})]\right]\right\}\right)\right] \\
    &\leq \varphi_{\Phi}\left(\frac{\lambda \cdot L \cdot \bar{L}}{N}; \alpha_{J + (i - 2)K}\right)
\end{align*}
where we use the fact that \(x^{(i -1)}\) marginally satisfies \(\PhiSI{\alpha_{J + (i - 2)K}}\) as remarked earlier.
Here, we place the implicit assumption that \(\frac{\lambda \cdot L \cdot \bar{L}}{N} < \Lambda\), which is satisfied by the previous assumption \(\frac{\lambda \cdot L}{N} < \Lambda\) since \(\bar{L} < 1\).
From this, we get the following bound for the product of \(\calA_{i}\) terms as
\begin{equation*}
    \prod_{i=2}^{N} \calA_{i} \leq \prod_{i=2}^{N} \varphi_{\Phi}\left(\frac{\lambda \cdot L \cdot \bar{L}}{N}; \alpha_{J + (i - 2)K}\right)~.
\end{equation*}
This results in the final bound
\begin{equation*}
    \calP_{N} \leq \varphi_{\Phi}\left(\frac{\lambda \cdot L}{N}; \beta\right)^{N - 1} \cdot \varphi_{\Phi}\left(\frac{\lambda \cdot L}{N}; \alpha_{J}\right) \cdot \prod_{i=2}^{N} \varphi_{\Phi}\left(\frac{\lambda \cdot L \cdot \bar{L}}{N}; \alpha_{J + (i -2)K}\right)~.
\end{equation*}
This results in the final bound for the probability as
\begin{multline*}
    \bbP\left(\widehat{F}(f; \calS_{N}) - \bbE[\widehat{F}(f; \calS_{N}) > t\right) \\\leq \inf_{\lambda ~:~ \lambda < \frac{\Lambda N}{L}} \varphi_{\Phi}\left(\frac{\lambda \cdot L}{N}; \beta\right)^{N - 1} \cdot \varphi_{\Phi}\left(\frac{\lambda \cdot L}{N}; \alpha_{J}\right) \cdot \prod_{i=2}^{N} \varphi_{\Phi}\left(\frac{\lambda \cdot L \cdot \bar{L}}{N}; \alpha_{J + (i -2)K}\right) \exp\left(-\lambda t\right)~.
\end{multline*}

We instantiate this with the choices of \(\varphi_{\Phi}\) and \(\Lambda\) highlighted in \cref{assump:mgf}.
\begin{description}
    \item [\(\Phi(t) = t\log(t)\): ] the function of \(\lambda\) in the infimum turns to
    \begin{equation*}
        \frakm(\lambda) := \exp\left(\frac{\beta \lambda^{2}L^{2}(N - 1)}{2N^{2}} + \frac{\alpha_{J}\lambda^{2}L^{2}}{2N^{2}} + \sum_{i=2}^{N} \frac{\lambda^{2}L^{2}\bar{L}^{2}\alpha_{J + (i - 2)K}}{2N^{2}} - \lambda t\right)~.
    \end{equation*}
    Taking the derivative w.r.t. \(\lambda\) and setting it to \(0\) results in
    \begin{equation*}
        \lambda \left(\beta (N - 1) + \alpha_{J} + \sum_{i=2}^{N} \bar{L}^{2} \alpha_{J + (i - 2)K}\right) = t \cdot \frac{N^{2}}{L^{2}}
    \end{equation*}
    This results in the infimum
    \begin{equation*}
        \exp\left(-\frac{N^{2}t^{2}}{2L^{2}(\beta(N - 1) + \alpha_{J} + \sum_{i=2}^{N} \bar{L}^{2}\alpha_{J + (i - 2)K})}\right)~.
    \end{equation*}
    To better understand the dependency on \(N\), we further bound this quantity from above.
    We specifically consider the denominator, and bound the partial series in the definition of \(\alpha_{J}, \alpha_{J + (i - 2)K}\) with the value of the infinite series (which exists on account of \(\bar{L} < 1\)).
    \begin{align*}
        \beta(N - 1) &+ \alpha_{J} + \sum_{i=2}^{N} \bar{L}^{2}\alpha_{J + (i - 2)K} \\
        &= \beta(N - 1) + \beta\sum_{i=0}^{J -1}\bar{L}^{2i} + \alpha_{0}\bar{L}^{2J} + \sum_{i=2}^{N} \bar{L}^{2}\left\{\beta \sum_{i'=0}^{J + (i - 2)K - 1} \bar{L}^{2i'} + \alpha_{0}\bar{L}^{2J + 2(i - 2)K}\right\} \\
        &\leq \beta(N - 1) + \frac{\beta}{1 - \bar{L}^{2}} + \alpha_{0} \bar{L}^{2J} + \bar{L}^{2}\sum_{i=2}^{N}\left\{\frac{\beta}{1 - \bar{L}^{2}} + \alpha_{0}\bar{L}^{2J} \cdot \bar{L}^{2(i - 2)K}\right\} \\ 
        &\leq \frac{\beta(N - 1)}{1 - \bar{L}^{2}} + \frac{\beta}{1 - \bar{L}^{2}} + \alpha_{0}\bar{L}^{2J}\left(1 + \frac{\bar{L}^{2}}{1 - \bar{L}^{2K}}\right) \\
        &= \frac{\beta \cdot N}{1 - \bar{L}^{2}} + \alpha_{0}\bar{L}^{2J}\left(1 + \frac{\bar{L}^{2}}{1 - \bar{L}^{2K}} \right) =: \frakQ(\alpha_{0}, \beta, \bar{L}, K, J, N)~.
    \end{align*}
    In summary, the infimum is bounded by \(\exp\left(-\frac{N^{2}t^{2}}{2L^{2}\frakQ(\alpha_{0}, \beta, \bar{L}, K, J, N)}\right)\).
\end{description}

\end{proof}

\subsubsection{Proof of \Cref{lem:tv-bounded,lem:kl-mgf,lem:renyi-moment}}
\label{sec:prf:bias-bounds}

\begin{proof}
We give the proofs of these lemmas collectively.
\paragraph{Proof of \Cref{lem:tv-bounded}} This lemma is a direct consequence of the dual representation of the total variation distance which states that for any two distributions \(\rho_{1}, \rho_{2}\),
\begin{equation*}
    \mathsf{TV}(\rho_{1}, \rho_{2}) = \sup_{g} \left\{ \bbE_{\rho_{1}}[g] - \bbE_{\rho_{2}}[g] ~:~ |g(x)| \leq 1 ~\forall~ x \in \bbR^{d}\right\}~.
\end{equation*}
Therefore,
\begin{equation*}
    \left|F^{\star}(f; \rho^{\star}) - F^{\star}(f; \pi^{\star})\right| \leq B \cdot \mathsf{TV}(\rho^{\star}, \pi^{\star})~.
\end{equation*}

For the other two lemmas, we can alternatively write
\begin{align*}
    F^{\star}(f; \rho^{\star}) - F^{\star}(f; \pi^{\star}) &= \bbE_{\rho^{\star}}[f] - \bbE_{\pi^{\star}}[f] \\
    &= \bbE_{\pi^{\star}}\left[\left(\frac{\rmd \rho^{\star}}{\rmd \pi^{\star}} - 1\right) \cdot f\right]~\numberthis\label{eq:bias-relative-density}
\end{align*}
\paragraph{Proof of \Cref{lem:kl-mgf}}

We use the fact that \(\bbE_{\pi^{\star}}\left[\frac{\rmd \rho^{\star}}{\rmd \pi^{\star}}\right] = 1\), and hence
\begin{equation*}
    \bbE_{\pi^{\star}}\left[\left(\frac{\rmd \rho^{\star}}{\rmd \pi^{\star}} - 1\right) \cdot f\right] = \bbE_{\pi^{\star}}\left[\frac{\rmd \rho^{\star}}{\rmd \pi^{\star}} \cdot (f - \bbE_{\pi^{\star}}[f])\right]~.
\end{equation*}
Now, we use \Cref{prop:product-phi-entropy} with \(\Phi(t) = t \log(t)\), \(f_{2} \leftarrow \frac{1}{\lambda} \cdot \frac{\rmd \rho^{\star}}{\rmd \pi^{\star}}\), and \(f_{1} \leftarrow e^{\lambda \cdot (f - \bbE_{\pi^{\star}}[f])}\) for an arbitrary \(\lambda > 0\).
Since \(\Phi'(t) = 1 + \log t\),
\begin{align*}
    \bbE_{\pi^{\star}}\left[(f_{1} - \bbE_{\pi^{\star}}[f_{1}]) \cdot f_{2}\right] &\leq \ent_{\pi^{\star}}\left[\frac{1}{\lambda} \cdot \frac{\rmd \rho^{\star}}{\rmd \pi^{\star}}\right] + \frac{1}{\lambda} \cdot \log \bbE_{\pi^{\star}}\left[\exp\left(\lambda \cdot (f - \bbE_{\pi^{\star}}[f])\right)\right] \\
    &= \frac{1}{\lambda} \cdot \ent_{\pi^{\star}}\frac{\rmd \rho^{\star}}{\rmd \pi^{\star}} + \lambda \cdot \frac{1}{\lambda^{2}} \cdot \log \bbE_{\pi^{\star}}\left[\exp\left(\lambda \cdot (f - \bbE_{\pi^{\star}}[f]\right)\right] \\
    &\leq \frac{1}{\lambda} \mathsf{KL}(\rho^{\star} \| \pi^{\star}) + \lambda^{p - 1} \cdot \sup_{\lambda > 0} \frac{1}{\lambda^{p}} \log \bbE_{\pi^{\star}}\left[\exp\left(\lambda \cdot (f - \bbE_{\pi^{\star}}[f]\right)\right]~.
\end{align*}
To obtain the tightest bound, we minimise the RHS w.r.t \(\lambda > 0\).
\begin{equation*}
    F^{\star}(f; \rho^{\star}) - F^{\star}(f; \pi^{\star}) \leq \left(\mathsf{KL}(\rho^{\star} \| \pi^{\star})\right)^{\frac{p - 1}{p}} \cdot \left(\sup_{\lambda > 0} \frac{1}{\lambda^{p}} \log \bbE_{\pi^{\star}}\left[\exp(\lambda \cdot (f - \bbE_{\pi^{\star}}[f]))\right] \right)^{\frac{1}{p}} \cdot \frac{p}{(p - 1)^{\frac{p - 1}{p}}}~.
\end{equation*}
To obtain a bound \(F^{\star}(f; \pi^{\star}) - F^{\star}(f; \rho^{\star})\), take \(f_{1} \leftarrow e^{\lambda \cdot (\bbE_{\pi^{\star}}[f] - f}\) and repeating the same arguments completes the proof.

\paragraph{Proof of \Cref{lem:renyi-moment}}

From \cref{eq:bias-relative-density}, we apply H\"{o}lder's inequality as
\begin{align*}
    \left|F^{\star}(f; \rho^{\star}) - F^{\star}(f; \pi^{\star})\right| &\leq \bbE_{\pi^{\star}}\left[\left|\frac{\rmd \rho^{\star}}{\rmd \pi^{\star}} - 1\right| \cdot |f|\right] \\
    &\leq \bbE_{\pi^{\star}}\left[\left|\frac{\rmd \rho^{\star}}{\rmd \pi^{\star}} - 1\right|^{q}\right]^{\frac{1}{q}} \cdot \bbE_{\pi^{\star}}[|f|^{\frac{q}{q - 1}}]^{\frac{q - 1}{q}}~.
\end{align*}
From \cref{prop:algebraic-identity}, we have
\begin{equation*}
    \bbE_{\pi^{\star}}\left[\left|\frac{\rmd \rho^{\star}}{\rmd \pi^{\star}} - 1\right|^{q}\right] \leq \bbE_{\pi^{\star}}\left[\left(\frac{\rmd \rho^{\star}}{\rmd \pi^{\star}}\right)^{q}\right] - 1 = \exp\left((q - 1) \sfD_{q}(\rho^{\star} \| \pi^{\star})\right) - 1~.
\end{equation*}
Substituting this in the bound for the bias completes the proof.
\end{proof}

%% file: appendix.tex
\appendix
\section{Definitions and facts excluded from the main text}
\label{app:sec:defs}

\begin{definition}[Lipschitz continuity]
A map \(\varphi : \bbR^{d} \to \bbR^{p}\) is said to be \emph{\(L\)-Lipschitz continuous} if for any two \(x, y \in \bbR^{d}\),
\begin{equation*}
    \|\varphi(x) - \varphi(y)\| \leq L \cdot \|x - y\|~.
\end{equation*}
\end{definition}

\begin{definition}[Sub-Gaussianity]
A random variable \(Z \in \bbR^{d}\) distributed according to \(\pi\) is said to be \emph{sub-Gaussian with variance proxy \(\sigma^{2}\)} if for any \(\lambda \in \bbR\) and unit vector \(u \in \bbR^{d}\)
\begin{equation*}
    \bbE_{\pi}[\exp\left\{\lambda \cdot \langle u, Z\rangle - \bbE_{\pi}[\lambda \cdot \langle u, Z\rangle]\right\}] \leq \exp\left(\frac{\lambda^{2} \cdot \sigma^{2}}{2}\right)~.
\end{equation*}
\end{definition}

\begin{proposition}
\label{prop:ibp}
Let \(F \in \cssf{\bbR^{d}}\) and \(G : \bbR^{d} \to \bbR\) be differentiable functions.
Then,
\begin{equation*}
    \int \nabla F(x) G(x) \rmd x = -\int \nabla G(x) F(x) \rmd x~.
\end{equation*}
\end{proposition}
\begin{proof}
    For any \(u \in \bbR^{d}\), consider
    \begin{align*}
        \left\langle u, \int \nabla F(x) G(x) \rmd x\right\rangle &= \int \langle u, \nabla F(x)\rangle G(x) \rmd x \\
        &= \int \langle G(x) u, \nabla F(x)\rangle \rmd x \\
        &\overset{(a)}= -\int F(x) \nabla \cdot (G(x)u)\rmd x\\
        &\overset{(b)}= -\int F(x) (\nabla \cdot (u) G(x) + \langle u, \nabla G(x)\rangle) \rmd x \\
        &= -\int F(x) \langle u, \nabla G(x)\rangle \rmd x \\
        &= -\left\langle u, \int F(x) \nabla G(x) \rmd x\right\rangle~.
    \end{align*}
    In step \((a)\), we use integration-by-parts and the fact that \(F \in \cssf{\bbR^{d}}\), which means that the boundary terms vanish, and in step \((b)\), we use the chain rule for the divergence.
\end{proof}

\begin{proposition}
\label{prop:youngs-square}
Let \(a, b \in \bbR\).
Then, for any \(\sfC > 0\),
\begin{equation*}
    (a + b)^{2} \leq (1 + \sfC) \cdot a^{2} + (1 + \sfC^{-1}) \cdot b^{2}
\end{equation*}
\end{proposition}
\begin{proof}
By expanding the square
\begin{equation*}
    (a + b)^{2} = a^{2} + b^{2} + 2ab \leq a^{2} + b^{2} + 2|a||b| \leq a^{2} + \sfC \cdot a^{2} + b^{2} + \sfC^{-1} \cdot b^{2}~,
\end{equation*}
where the final step applies Young's inequality for products (or AM-GM inequality).
\end{proof}

\begin{proposition}
\label{prop:f-divergence-coupling}
Let \(d_{\phi}\) be the \(\phi\)-divergence for a convex function \(\phi\), defined between two distributions \(\mu_{1}, \mu_{2} \in \mathcal{P}_{\mathrm{ac}}(\bbR^{d})\) as
\begin{equation*}
    d_{\phi}(\mu_{1} \| \mu_{2}) := \bbE_{\sfx \sim \mu_{2}}\left[\phi\left(\frac{\rmd \mu_{1}}{\rmd \mu_{2}}(\sfx)\right)\right]~.
\end{equation*}
If \(\mu_{1}, \mu_{2}\) are mixtures \emph{(\cref{eq:mixture-density})} of \(\{P_{y}\}_{y \in \bbR^{d}}\) generated by mixing distributions \(\rho_{1}, \rho_{2} \in \calP_{\mathrm{ac}}(\bbR^{d})\) respectively, then for any coupling \(\gamma\) of \(\rho_{1}, \rho_{2}\) we have
\begin{equation*}
    d_{\phi}(\mu_{1} \| \mu_{2}) \leq \bbE_{(\sfy_{1}, \sfy_{2}) \sim \gamma}\left[d_{\phi}\left(P_{X | Y = \sfy_{1}} \| P_{X | Y = \sfy_{2}}\right)\right]~.
\end{equation*}
\end{proposition}
\begin{proof}
For any coupling \(\gamma\) of \(\rho_{1}, \rho_{2}\), we have \(\mu_{i} = \bbE_{\sfy_{i} \sim \rho_{i}}\left[P_{X | Y = \sfy_{i}}\right] = \bbE_{(\sfy_{1}, \sfy_{2}) \sim \gamma}\left[P_{X | Y = \sfy_{i}} \right]\) for \(i \in \{1, 2\}\).
Since \(\phi\) is convex, \(d_{\phi}(\cdot \| \cdot)\) is jointly convex in its arguments, which leads to
\begin{equation*}
    d_{\phi}(\mu_{1} \| \mu_{2}) \leq \bbE_{(\sfy_{1}, \sfy_{2}) \sim \gamma}\left[d_{\phi}(P_{X | Y = \sfy_{1}} \| P_{X | Y = \sfy_{2}})\right]~.
\end{equation*}
\end{proof}

\begin{proposition}
\label{prop:algebraic-identity}
Let \(t > 0\) and \(q \geq 2\).
Then
\begin{equation*}
    |t - 1|^{q} \leq t^{q} - 1 - q(t - 1)~.
\end{equation*}
\end{proposition}
\begin{proof}
    Note that when \(q = 2\), this is precisely expanding \((t - 1 + 1)^{2}\).
    There are two cases: when \(t \geq 1\) and when \(t \in [0, 1)\).
    A key fact we will use is the superadditivity of convex functions over non-negative reals that take the value \(0\) at \(0\).
    \begin{itemize}
        \item When \(t \geq 1\), consider the function \(g(t) := t^{q} - (t - 1)^{q} - 1 - q(t - 1)\).
        Its derivative is \(qt^{q - 1} - q(t - 1)^{q - 1} - q\).
        Since \(s \mapsto s^{q - 1}\) over \([0, \infty)\) is convex and maps \(0\) to \(0\), we have superadditivity: \(t^{q - 1} = (t - 1 + 1)^{q - 1} \geq (t - 1)^{q- 1} + 1\), which implies that \(g(t)\) is increasing.
        Since \(g(1) = 0\), we have the required inequality for \(t \geq 1\).</li>
    
        \item When \(t \in [0, 1)\), consider the function \(g(t) := (1 - t)^{q} - t^{q} + 1 + q(t - 1)\).
        Its derivative is \(-q(1 - t)^{q - 1} - qt^{q - 1} + q\).
        Again, by superadditivity \(1 = (1 - t + t)^{q - 1} \geq (1 - t)^{q - 1} + t^{q - 1}\).
        Therefore, \(g(t)\) is a decreasing function, and \(g(1) = 0\), which implies the required statement for \([0, 1)\).
    \end{itemize}
\end{proof}